\documentclass[12pt]{article}
\usepackage{amsfonts, amsmath}
\usepackage{amssymb}
\usepackage{amsmath,amscd}
\usepackage{lipsum}
\usepackage[pdftex]{graphicx}

\usepackage[margin=25pt,font=small,labelfont=bf]{caption}
\usepackage{subcaption}

\newcommand{\bna}{\begin{eqnarray}}
\newcommand{\ena}{\end{eqnarray}}
\newcommand{\ba}{\begin{eqnarray*}}
\newcommand{\ea}{\end{eqnarray*}}

\newcommand{\bs}[1]{}
\newtheorem{theorem}{Theorem}[section]
\newtheorem{corollary}{Corollary}[theorem]
\newtheorem{lemma}[theorem]{Lemma}
\numberwithin{figure}{section}
\newtheorem{proposition}{Proposition}[section]
\newtheorem{definition}{Definition}[section]
\numberwithin{equation}{section}
\DeclareMathOperator{\Dim}{Dim}
\DeclareMathOperator{\rank}{rank}

\def\Cal{\mathcal C}
\def\CalM{\mathcal M}
\def\CalH{\mathcal H}
\def\CalA{\mathcal A}
\def\CalP{\mathcal P}
\def\A{\mathcal A}

\def\R{\mathbb R}

\def\RP{\mathbb R \mathbb P }

\def\p{{\bf p}}
\def\b{{\bf b}}
\def\pn{{\bf p =(p_1, \dots, p_n) }}
\def\q{{\bf q}}
\def\qn{{\bf q =(q_1, \dots, q_n) }}
\def\v{{\bf v}}
\def\w{{\bf w}}
\def\e{{\bf e}}
\def\r{{\bf r}}

\def\x{{\bf x}}
\def\int{{\text{int}}}
\begin{document}

\title{Iterative Universal Rigidity}

\author{Robert Connelly and  
Steven J. Gortler}
\maketitle

\begin{abstract}  
A bar framework determined by a finite graph $G$ and configuration
$\pn$ in $\R^d$ is \emph{universally rigid} if it is rigid in any
$\R^D \supset \R^d$. We provide a characterization of
universally rigidity
for any graph $G$ and any configuration $\p$ in terms of a sequence of
affine subsets of the space of configurations.  This corresponds to a
facial reduction process for closed finite dimensional convex cones.
\end{abstract}

\section{Introduction}\label{sect:introduction}

\subsection{Basic Definitions}\label{subsect:basics}
Given a configuration $\pn$ of $n$ points in $\R^d$, and a finite
graph $G$, without loops or multiple edges, on those $n$ points one
can ask the natural and fundamental question:
is there another configuration $\qn$ in $\R^d$, where the distance
between $\p_i$ and $\p_j$, is the same as the distance between $\q_i$ and
$\q_j$ when $\{i,j\}$ is an edge of $G$?  When this happens we say
that $(G, \p)$ is \emph{equivalent} to $(G, \q)$.  (Traditionally
$G(\p)$ and $G(\q)$ is the notation used for $(G, \p)$ and $(G, \q)$,
which are called \emph{(bar) frameworks}, but we break that tradition
here.) Of course, if there is a congruence between $\p$ and $\q$, they
are called \emph{trivially equivalent} or \emph{congruent},
since all pairs of distances are the
same.

The following are a sequence of ever stronger rigidity
properties of frameworks, where $(G,\p)$ is a framework on $n$
vertices in $\R^d$.
\begin{itemize}
\item If all the frameworks $(G, \q)$ in $\R^d$ equivalent to $(G,
  \p)$ and sufficiently close to $(G, \p)$ are trivially equivalent to
  $(G, \p)$ we say that $(G, \p)$ is \emph{locally rigid in $\R^d$}
  (or just \emph{rigid in $\R^d$}).
\item If all the frameworks $(G, \q)$ in $\R^d$ equivalent to $(G,
  \p)$ are congruent to $(G, \p)$ we say that $(G, \p)$ is
  \emph{globally rigid in $\R^d$}.
\item If all the frameworks $(G, \q)$ in any $\R^D \supset \R^d$
  equivalent to $(G, \p)$ are trivially equivalent to $(G, \p)$,
 we say
  that $(G, \p)$ is \emph{universally rigid}.
\end{itemize}

\subsection{Main Result}\label{subsect:result}

It is well known that the existence  of a certain kind of ``stress'' matrix 
associated with a specific framework is sufficient to prove its  universal rigidity~\cite{Connelly-energy}. 
It is also known that when a ``generic'' framework
is universal rigidity, it is also
necessary for  it to have this type of associated
stress matrix~\cite{Gortler-Thurston2}. But there do exist  special frameworks that are universally
rigid while not possessing such a matrix.
In this paper, we propose a new criterion in terms of a certain ``sequence
of stress matrices'' which gives a complete (necessary and sufficient)
characterization of universal rigidity for any specific framework in
any dimension of any graph.

The validity of this certificate can be checked efficiently and
deterministically in the real computational mode of~\cite{Smale-NP}. 
We need to use a real model, since even if $\p$ is described using
rational numbers, the stress matrix might have irrational entries.
As such, this means
that universal rigidity is in the class \emph{NP} under this real
computational model. Note that universal rigidity is
clearly in \emph{CO-NP} under real valued computation since the non
universal rigidity of a framework $\p$ can always be certified by a
providing an equivalent non-congruent framework $\q$.

The main result will be explained in Section
\ref{sect:affine-constraints} and Theorem \ref{thm:main}.  We will
derive our results in a self-contained manner, but note that
technically, what we have is really a thinly disguised version of a
known technique called ``facial reduction'' which is used to analyze
convex cone programs \cite{Borwein-Wolkowicz}.  The connection is
explained explicitly in Section \ref{sect:facial}.

\subsection{Relation to other forms of rigidity}\label{subsect:others}

Given $(G, \p)$, testing for local or global rigidity is
known to be a hard computational problem~\cite{Abbot-Hard, saxe}.
Fortunately, this is not the end of the story.
For local and global rigidity, the problems become much easier if we
assume that $\p$ is generic.  (We say that a configuration $\p$ is
\emph{generic in $\R^d$} if all the coordinates of all the points of
$\p$ are algebraically independent over the rationals.  This means, in
particular, there can be no symmetries in the configuration, no three
points are collinear for $d \ge 2$, etc).  Local and global rigidity
have efficient randomized algorithms  under the
assumption that the configuration is generic, (and for $d=1$ or
$d=2$, there are even purely combinatorial polynomial-time algorithms).  See
\cite{Connelly-rigidity, Connelly-global, Gortler-Thurston,
  Whiteley-survey} for information about all of these concepts.  In
particular, both local and global rigidity in $\R^d$ are generic
properties of a graph $G$. That is, either all generic frameworks are
rigid, or none of them are, and so these properties only depend on the graph $G$ and not on the configuration $\p$.

One justification for assuming that a configuration
is generic is that in any region, the generic configurations form a
set of full measure.  In other words, if a configuration is chosen
from a continuous distribution, with probability one, it will be
generic, and with any physical system, there will always be some
indeterminacy with respect to the coordinates.  But the problem is
that special features of a particular configuration, such as symmetry,
collinearity, overlapping vertices, etc, may be of interest and they
are necessarily non-generic. In this paper we do not want to restrict
ourselves to generic frameworks.

In order to test for local rigidity of a specific non-generic framework
there is a natural sufficient condition to use,
namely \emph{infinitesimal rigidity}. This says that in
$\R^d$ (for $n \ge d$) the rank of the rigidity matrix $R(\p)$ is
$nd-d(d+1)/2$, where $R(\p)$ is an $m$-by-$nd$ (sparse) matrix with
integer linear entries, where $m$ is the number of \emph{members}
(another name for the bars) as defined in Section
\ref{sect:measurement}.  See also \cite{Whiteley-survey}, for
example. Infinitesimal rigidity of $(G,\p)$ can
can be computed
efficiently~\cite{Whiteley-survey}.

Infinitesimal rigidity is simply a linearized form of local rigidity
and thus is a very natural sufficient condition to use for
testing the local rigidity of $(G,\p)$. In fact, the matrix
test for infinitesimal rigidity 
is central to the determination of generic local rigidity for $G$
just described.
In contrast, we do not have such a natural sufficient condition to use
for global rigidity. Indeed, the particular matrix test used to 
compute generic global rigidity for the graph
$G$ does not give us information about the
global rigidity of any specific framework $(G,\p)$~\cite{Connelly-global}.

Thus, 
in order to test for global rigidity of a specific non-generic framework,
we often resort to ``stronger'' conditions;
perhaps the most usable such sufficient condition is, 
universal rigidity.   In this context, one can choose the ambient  dimension  $D$ to be, say $n-1$ with no loss in generality.  As such, understanding universal rigidity
can be essential to determining global rigidity, and it is the
focus of this paper.

\subsection{Complexity Issues}\label{subsect:complexity}

The theoretical complexity of testing universal rigidity for $(G,\p)$,
(even when $\p$ is given by integer-valued input)
is technically unknown. There are no known hardness results, nor
are there any provably correct efficient algorithms.
One can pose the problem of
universal rigidity
in the language of semi-definite programing (SDP)~\cite{Ye-So}. 
Unfortunately, the complexity for
for conclusively deciding an SDP feasibility problem is itself
unknown~\cite{Ramana}. 

In practice, one can use a numerical (say interior point) SDP solver
for these problems. Roughly speaking, 
this can efficiently find a framework with an affine
span of dimension $n-1$ (the highest possible dimension) that is
within $\epsilon$ of being equivalent to the given framework. 
If this framework
appears to ``almost'' have an affine span of dimension $d$, and
appears to 
be ``very close'' to 
the input $\p$, then we have strong ``evidence'' for universal
rigidity. But it is unclear how to use this 
to make a determination with provable correctness properties.  In effect,
this means, in the case with imprecise input, that the problem to determine whether the framework is universally rigid 
cannot be solved because there is not enough information in the input to be able to 
solve it.

An exasperating issue is that there can be great 
sensitivity between errors in achieving desired edge
lengths (which are what we get when using an SDP solver)
and errors in the resulting configuration.
Figure~\ref{fig:iterated-rigid} shows a framework (with pinned vertices)
that is universally rigid in $\R^2$. 
We will see that this can be verified using methods
described in this paper.
If the lengths in Figure \ref{fig:iterated-rigid} are all increased by
less that $0.5\%$, Figure \ref{fig:iterated-rigid-2} shows the
resulting realization in the plane.  Note that this slightly perturbed
framework is far from universally rigid.
Here we see  that a very small error in the
numerical calculation of the lengths of the members can lead to a very
large perturbation of the resulting configuration, and, indeed, the
decision as to universal rigidity may  be incorrect.

\begin{figure}
\centering
\begin{minipage}{.51\textwidth}
  \centering
  \includegraphics[width=.69\linewidth]{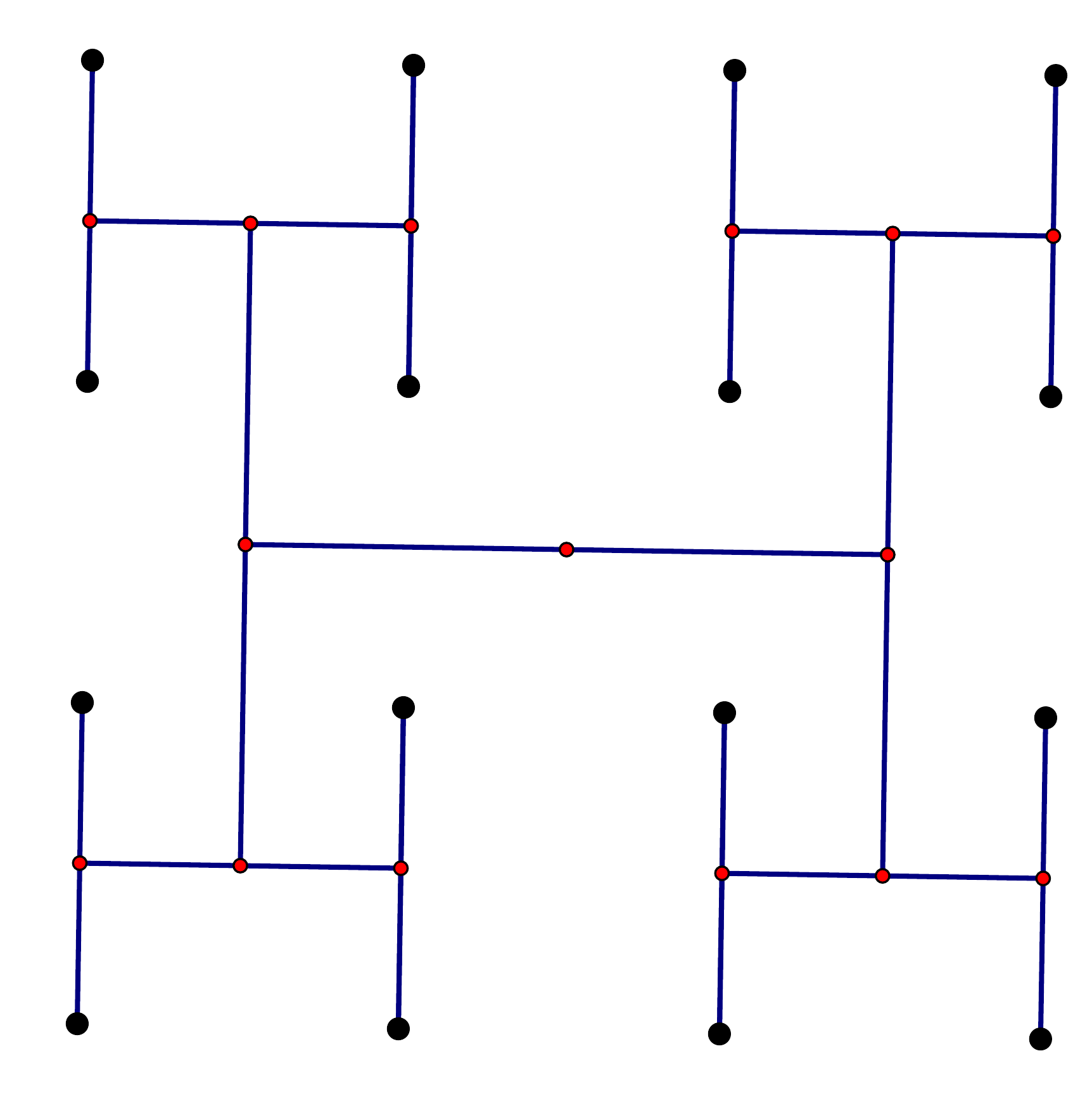}
  \captionof{figure}{The large black vertices are pinned to the plane, and the whole framework is universally rigid as 
in Corollary \ref{cor:main}.}
    \label{fig:iterated-rigid}
\end{minipage}%
\begin{minipage}{.51\textwidth}
  \centering
  \includegraphics[width=.69\linewidth]{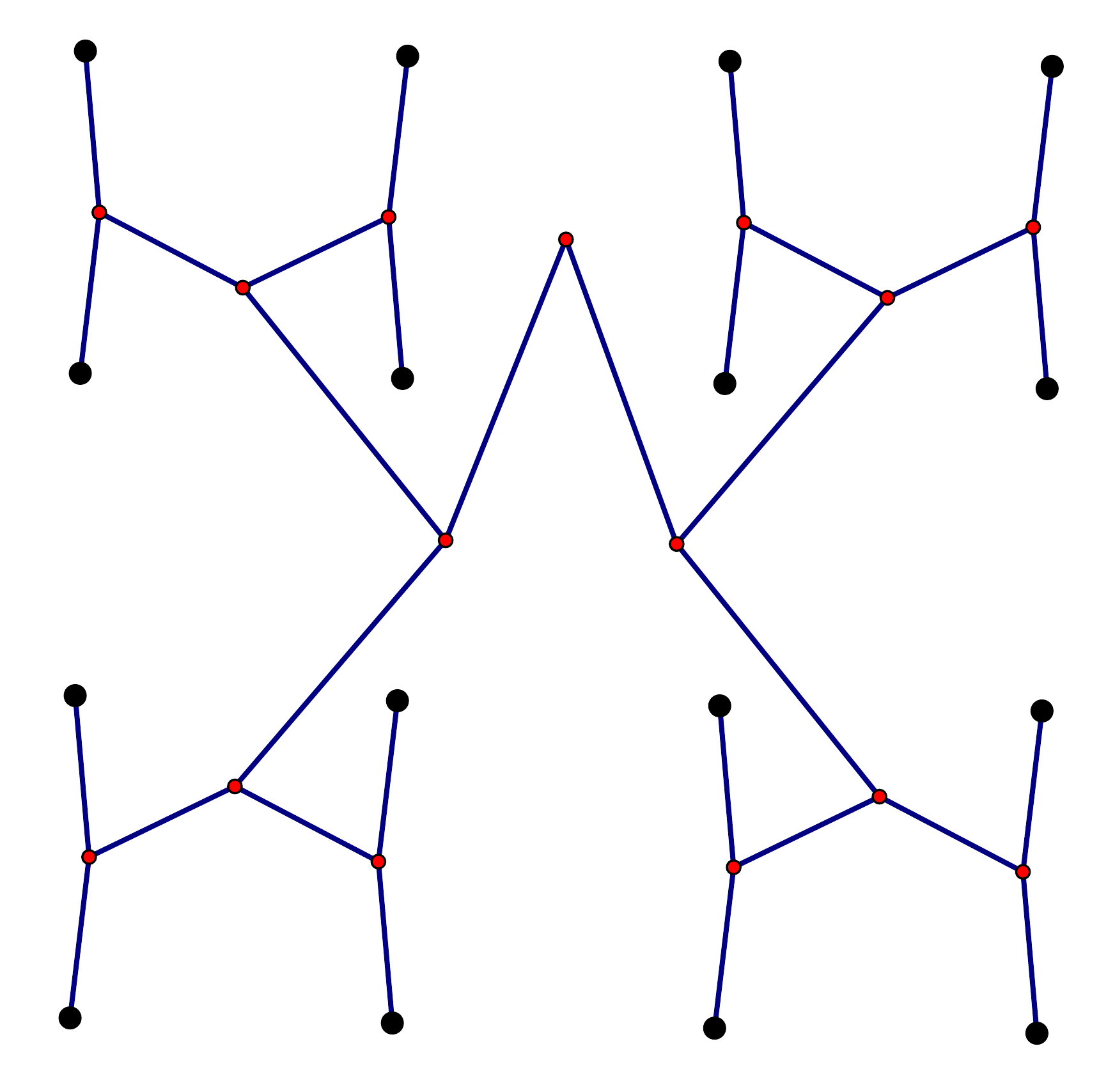}
  \captionof{figure}{This is the same framework as in Figure \ref{fig:iterated-rigid} but with the lengths of the bars increased by less than $0.5\%$.}
    \label{fig:iterated-rigid-2}
\end{minipage}
\end{figure}

\subsection{Certificates}\label{subsect:certificate}

The lack of conclusive algorithms for universal rigidity brings
us, finally, to the topic of ``sufficient certificates'' for universal rigidity.
In this paper we show that 
there is a kind of sufficient certificate that must exist 
for any universally rigid framework. This certificate is described by
a sequence of real-valued matrices and
can be 
verified efficiently using real computation.

Note,
that we do not claim that given a universally rigid framework, this
certificate can always be found efficiently.  
But, as we describe below in Section~\ref{sect:examples},
there are many cases where 
we can systematically find the appropriate certificates for the universal
rigidity of $(G,\p)$. 
We also discuss other cases 
where we have at least 
a positive probability of finding the certificate.

Looking again at the situation of
Figure~\ref{fig:iterated-rigid} 
and Figure~\ref{fig:iterated-rigid-2}, we see that universal rigidity
itself
can be a fragile property, that is destroyed 
(along with its sufficient certificates)
by any errors in the
description of $\p$. Given our new characterization of universal rigidity, 
we suggest that when exploring and designing frameworks that we wish to
be universally rigid, it may be best to explicitly maintain the
appropriate certificates as part of the representation and description
of $(G,\p)$.

\section{Stress}\label{sect:stress}

The central tool we will use to analyze universal rigidity is the concept 
of a stress.
\begin{definition}A \emph{stress}
associated to a graph $G$  is a scalar $\omega_{ij}=\omega_{ji}$
assigned to each edge $\{i,j\}$ of $G$.  Call the vector ${\bf
  \omega} = (\dots, \omega_{ij}, \dots)$, the \emph{stress} vector. 
\end{definition}

We can
suppress the role of $G$ here by simply requiring that $\omega_{ij}=0$
for any non-edge $\{i,j\}$ of $G$.  (One should also be careful not
to confuse the notion of stress here with that used in structure
analysis, in physics or in engineering.  There, stress is defined as a
force per cross-sectional area.  In the set-up here, there are no
cross-sections; the scalar $\omega_{ij}$ is better interpreted as a
force per unit length.)  

Since we will be concerned with configurations in an arbitrarily high
dimension, we will fix a large dimension $D$, which can effectively be
taken to be $n$ if our framework has $n$ vertices.  When we are given
a particular configuration, we generally will assume it is realized in
$\R^D$.  We can describe a configuration $\p$ in $\R^D$ using
coordinates using a single vector in $\R^{Dn}$.  Of course, for the
purposes of deciding universal rigidity and some of the other concepts
defined here, there is no reason to restrict the configurations to lie
some particular Euclidean space $\R^D$.  But it is clear that once the
ambient dimension $D$ is greater than $n$, any configuration in any
higher dimension is congruent to one in $\R^D$, and it will be
convenient to consider configurations in dimensions larger than $n$.
In order to define a finite dimensional space of configurations
appropriate for universal rigidity, though, it is useful to restrict
just to those configurations in $\R^D$, and if a construction pops out
of $\R^D$, we can always rotate it back in to $\R^D$.

Given a stress, we can measure the energy of a configuration:
Let $\omega=(\dots, \omega_{ij}, \dots)$ be a stress for a graph $G$ and let
$\pn$ be a configuration in $\R^D$. 
\begin{definition}We define the \emph{stress-energy}
associated to $\omega$ as
\begin{equation}\label{eqn:stress-energy}
E_{\omega}(\p):=\sum_{i<j} \omega_{ij} (\p_i-\p_j)^2,
\end{equation}
where the product of vectors is the ordinary dot product, and the
square of a vector is the square of its Euclidean length. 
\end{definition}
Regarding
the stress $\omega$ as fixed constants, $E_{\omega}$ is a quadratic
form defined on vectors in $\R^{Dn}$, 
it is
easy to calculate that the configuration $\p$ is a critical point for
$E_{\omega}$ when, for each vertex $i$ of $G$,

\begin{equation}\label{eqn:equilibrium}
\sum_j \omega_{ij} (\p_j-\p_i) = 0.
\end{equation}

\begin{definition}When Equation (\ref{eqn:equilibrium}) holds, we say that the stress $\omega$ is an \emph{equilibrium stress for the configuration} $\p$.
We also say that $\p$ is in \emph{equilibrium with respect to $\omega$}.
\end{definition}

It is useful to represent a stress in matrix form:
The $n$-by-$n$ \emph{stress matrix} $\Omega$ associated to the stress
$\omega$ is defined by making the $\{i,j\}$ entry of $\Omega$ be
$-\omega_{ij}$ when $i \ne j$, and the diagonal entries of $\Omega$
are such that the row and column sums of $\Omega$ are zero.

It is easy to see that with respect to the standard basis of $\R^{Dn}$,
the matrix of $E_{\omega}$ is $\Omega \otimes I^D$, where $I^D$ is the
$D$-by-$D$ identity matrix and $\otimes$ is the matrix Kronecker
product.  Note that although $E_{\omega}$ is defined over the high
dimensional space $\R^{nD}$, its being PSD only depends only on
$\Omega$, and its rank only depends on the rank of $\Omega$ and $D$.

If $\p$ is a configuration in $\R^d$
with an equilibrium stress $\omega$,  it
is easy to check that for any affine map of $a: \R^d \rightarrow
\R^D$, the configuration $a(\p)$ defined by $\p_i \rightarrow
a(\p_i)$, for all $i$, is also an equilibrium configuration with
respect to $\omega$.  

\begin{definition}\label{def:conic-infinity}
We say that a configuration $\p$ is
\emph{universal with respect to the stress $\omega$} if $\p$ is in
equilibrium with respect to $\omega$, and any other configuration $\q$
in $\R^D$ which is, also, in equilibrium with respect to $\omega$, is
such that $\q$ is an affine image of $\p$.
\end{definition}

\begin{definition}
For a configuration $\pn$ we regard each $\p_i$ as a column vector
in $\R^D$, as we define the $D$-by-$n$ \emph{configuration matrix of $\p$} as 
\[
P = 
\begin{bmatrix}
\p_1 & \p_2 & \dots & \p_n
\end{bmatrix}.
\]
\end{definition}

Then it is easy to check that the equilibrium condition for a given stress is 
\[
P\,\Omega  = 0,
\]
where $\Omega$ is the stress matrix for the stress $\omega$.  

The following is easy to check and is in \cite{Connelly-energy}.
See also Lemma \ref{lem:metric-universality} for a general universal construction.

\begin{proposition}\label{prop:universal-configuration}  
Given a stress $\omega$, let $\p$ be any configuration that is in
equilibrium with respect to $\omega$ and whose affine span is of
maximal dimension over all such configurations. Let this affine
span have dimension $d$.  Then $\p$ is universal with respect to
$\omega$ and the rank of $\Omega$ is $n-d-1$.
\end{proposition}

\section{The conic at infinity}\label{sect:conic}

In a sense, an equilibrium stress can only make distinctions ``up to
affine motions" as seen in Proposition
\ref{prop:universal-configuration}.  For rigidity questions, we would
like to know when the affine motions can be restricted to Euclidean
congruences.  

\begin{definition} We say that $\v= \{\v_1, \dots, \v_m\}$, a finite
collection of non-zero vectors in $\R^d$, lie on a \emph{conic at
  infinity} if when regarded as points in real projective $(d-1)$
space $\RP^{d-1}$, they lie on a conic. 
\end{definition}

  This means that there is a
non-zero $d$-by-$d$ symmetric matrix $A$ such that for all $i = 1,
\dots, m$, $\v_i^tA\v_i=0$, where $()^t$ is the transpose operation.
The following shows how affine motions can be non-trivial flexes of a
framework.
\begin{definition}A \emph{flex} of a framework $(G,\p)$ is a continuous
motion $\p(s)$, $0 \le s \le 1, \p(0)=\p$, where $\p(s)$ is equivalent
to $\p$.  It is non-trivial if $\p(s)$ is not congruent to $\p$ for
all $s>0$.  If $\p(s)=A(s)\p(0)$, where $A(s)$  is an affine function of Euclidean space, then we say $\p(s)$ is an  \emph{affine flex}. 
\end{definition}
\begin{proposition}\label{prop:affine-flex}  
A framework $(G,\p)$ in $\R^d$, with $d$-dimensional affine span, has a non-trivial affine flex if and
only if it has an equivalent non-congruent affine image in $\R^d$ if
and only if the member directions $\{\p_i-\p_j\}_{\{i,j\}\in E(G)}$
lie on a conic at infinity, where $E(G)$ are the edges of $G$.
\end{proposition}
See \cite{Connelly-energy, Shaping-space} for a simple proof of this property.
Note that in the plane, the conic lies in $\RP^1$, which consists of
two points or one point.  So affine motions of a framework can only
occur when the edge directions lie in two possible directions.

\section{The fundamental theorem}\label{sect:fundamental}

The major tool used for  proving universal rigidity is the following.  (See \cite{Connelly-energy}.)

\begin{theorem} \label{thm:fundamental}  Let $(G,\p)$ be a framework whose affine span of $\p$ is all of $\R^d$, with an equilibrium stress $\omega$ and stress matrix $\Omega$.  Suppose further
\begin{enumerate}
\item  $\Omega$ is positive semi-definite (PSD).
\item  The configuration $\p$ is universal with respect to the stress $\omega$.  (In other words, the rank of $\Omega$ is $n - d - 1$.)
\item  The member directions of $(G,\p)$ do not lie on a conic at infinity.
\end{enumerate}
Then $(G,\p)$ is universally rigid.
\end{theorem}

The idea is that $E_\omega(\p)$ only depends on the edge lengths of $\p$,
and so any configuration $\q$ equivalent  to $\q$ must have zero energy.
Since $E_\omega$ is PSD, this forces such a $\q$ to have coordinates
in the kernel of $\Omega$ and thus $\q$ to be  an affine image of $\p$.
Thus by Proposition \ref{prop:affine-flex},
their member directions must lie on a conic
at infinity. So Condition \ref{condition:conic} implies that $(G,\p)$ is universally rigid.


\begin{definition}  If all three conditions of Theorem \ref{thm:fundamental} are met we
say that the framework $(G,\p)$ is \emph{super stable}.  
\end{definition}

There are many
instances of such frameworks. For example, the rigid tensegrities of
\cite{Connelly-Terrell} are super stable in $\R^3$, where the
number of edges of $G$ is $m =2n$, and $n$ is the number of vertices.
Theorem \ref{thm:fundamental} is the starting point for most of our results in this
paper, where this result will be generalized significantly.  

Given such a matrix $\Omega$ and $(G,\p)$ as real valued input, one can
efficiently verify (under, say a real-model of computation)
that $\Omega$ is PSD and that it is an equilibrium stress matrix for $(G,\p)$.

We note, in passing, the following result  in
\cite{Alfakih-Ye-general-position} which replaces the conic condition
with a more natural one.

\begin{definition}A configuration $\pn$ in $\R^d$
is in \emph{general position} if no $k$ points lie in a $(k-1)$-dimensional
affine space for $1 \le k \le d$.
\end{definition}

\begin{theorem}\label{Alfakih-general}
If Conditions \ref{condition:positive} and \ref{condition:rank}
hold in Theorem \ref{thm:fundamental} and Condition
\ref{condition:conic} is replaced by the assumption that the
configuration $\p$ is in general position, then Condition
\ref{condition:conic} still holds and $(G,\p)$ is super stable.
\end{theorem}

This natural question is whether the conditions of Theorem
\ref{thm:fundamental} are necessary for universal rigidity.
The answer \emph{in the generic case} is in the affirmative.
The following is from
\cite{Gortler-Thurston2}:
\begin{theorem}\label{thm:Gortler-Thurston2}
A universally rigid framework $(G, \p)$, with $\p$  generic in $E^d$ and having $n \ge d+2$ vertices, 
has 
a PSD equilibrium stress matrix with rank $n -d- 1$.
\end{theorem}

This result does not hold for non-generic frameworks (even in general position).
For example, see the universally rigid framework in Figure~\ref{fig:general-position}.
In this paper, we will describe a (weaker) sufficient condition that
is also necessary for universal rigidity for all frameworks.
\begin{figure}[here]
    \begin{center}
        \includegraphics[width=0.4\textwidth]{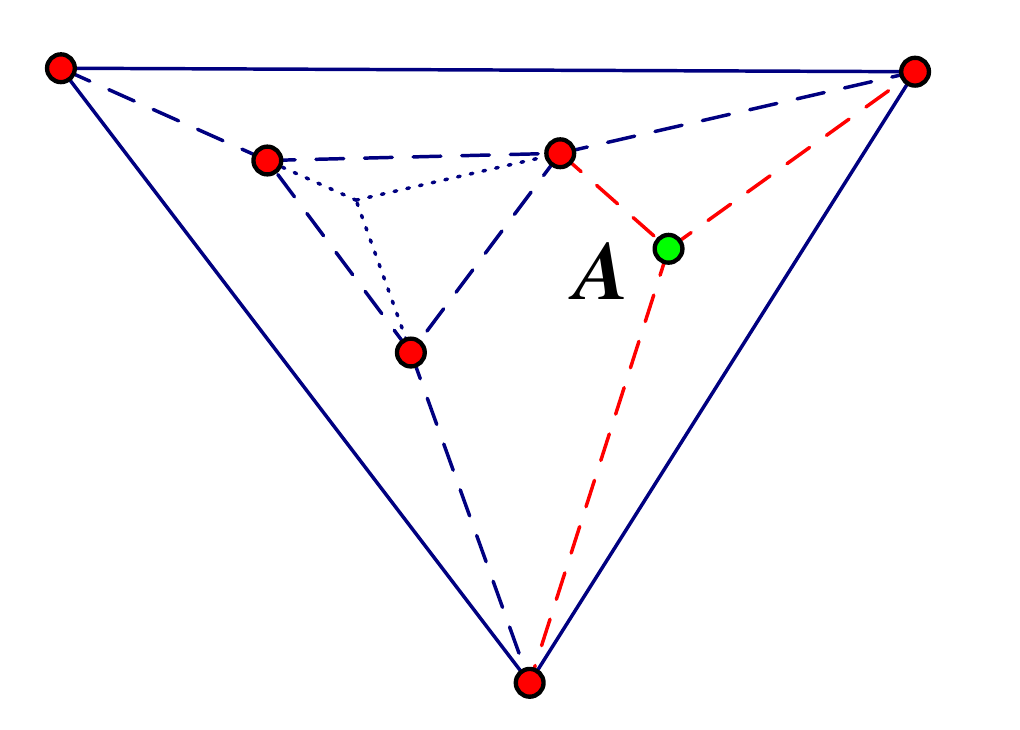}%
        \end{center}
    \caption{This is a framework, where the vertices are all in
      general position, there is only a one-dimensional space of
      equilibrium stresses, and the associated stress matrix does not
      have maximal rank.  The stresses on the members at the vertex
      $A$ must be all zero. The dotted lines extending the members coming
      from the vertices of the outside triangle meet at a point and
      are not part of the framework, as shown.  As described in this paper,
we will use multiple levels of stresses.
In this figure and
      later ones, the first level stresses and the corresponding
      members are colored in dark blue, the next level in red and the
      third level in green.}
    \label{fig:general-position}
    \end{figure}

\section{Dimensional rigidity}\label{sect:dimensional-rigidity}

In \cite{Alfakih-dim-rigidity} a notion called dimensional rigidity
is introduced.
This is closely related to, but
distinct from, universal rigidity.
Our main result can be best understood in terms of dimensional
rigidity, first.  Then we can derive the appropriate statements about
universal rigidity.

\begin{definition}  We say that a
framework $(G,\p)$ in $\R^d$, with affine span of dimension $d$, is
\emph{dimensionally rigid}  in $\R^d$ if every framework $(G,\q)$ equivalent to
$(G,\p)$ has an affine span of dimension no greater than $d$.
\end{definition}

(One might better call this concept \emph{dimensionally maximal}, since a
dimensionally rigid framework may not even be locally rigid, but we
refrain from that indulgence.) 

In many applications, one often wants to find the {\bf
  minimum dimension} for a graph $(G, \e)$ with given edge lengths $
\e= \{\dots, e_{ij}, \dots\}$, so the concept of the maximal dimension
seems backwards from what is normally desired.  For example, finding
the minimal dimension of $(G, \e)$ is the point of
\cite{Belk-Connelly, Laurent-Gram}.  Nevertheless, dimensional rigidity is quite
relevant for universal rigidity.

It is clear that if a framework $(G,\p)$ is universally rigid in
$\R^d$, then it is dimensionally rigid in $\R^d$, but we shall see
several examples of non-rigid dimensionally rigid frameworks.  
Such cases always occur due to a conic at infinity, (in which case,
the framework is not even locally rigid).
For
example, two bars, with a single vertex in common, is dimensionally
rigid in the plane, but it is flexible, i.e. not rigid, in the plane.

An important connection between dimensional rigidity and universal
rigidity is the following.  (This is proved in \cite{Alfakih-dim-rigidity}, but we provide a more direct proof here.)
\begin{theorem}\label{thm:affineness} 
If a framework $(G,\p)$ with $n$ vertices in $\R^d$ is dimensionally
rigid in $\R^d$, and $(G,\q)$ is equivalent to $(G,\p)$, then $\q$ is
an affine image of $\p$.
\end{theorem}
\begin{proof} 
Suppose that $h: \p \rightarrow \q$ is the correspondence between the
configurations.  Consider the graph of this correspondence $\Gamma(h)
= \{ (\p_i, \q_i)\}_{i=1,\dots, n} \subset \R^d \times \R^D$, where
$D$ is sufficiently large to contain $\q$.  It is easy to check (See
\cite{Bezdek-Connelly-two-distance} 
or  the proof of 
Lemma \ref{lem:metric-universality} below) 
that
$\frac{1}{\sqrt{2}}\Gamma(h)$ is equivalent to $\p$ and $\q$.  
Thus
there is a $d$-dimensional affine hyperplane that contains
$\frac{1}{\sqrt{2}}\Gamma(h)$.  This implies that $\q$ is an affine
image of $\p$. $\Box$
\end{proof}

A key consequence of Theorem \ref{thm:affineness} shows that universal rigidity can be determined from dimensional rigidity and Property \ref{condition:conic}.) of Theorem \ref{thm:fundamental}.

 \begin{corollary}  A framework $(G,\p)$ with $n$ vertices in $\R^d$ is universally rigid if and only if it is dimensionally rigid and the edge directions do not lie on a conic at infinity. 

\end{corollary}

 One result that follows from the proof of Theorem \ref{thm:fundamental} from \cite{Connelly-energy} is the following.

\begin{theorem}\label{thm:Alfakih} If a framework $(G,\p)$ with $n$ vertices in $\R^d$ has an equilibrium stress with a PSD stress matrix of rank $n-d-1$, then $(G,\p)$ is dimensionally rigid in $\R^d$.
\end{theorem}

See \cite{Alfakih-dim-rigidity} for similar conditions for dimensional
rigidity.  This just says that the configuration $\p$ is universal
with respect to the given stress.  The only other possible equivalent
configurations of $(G,\p)$, in this case, are affine linear images,
which do not raise dimension.  

Since universal rigidity implies dimensional rigidity, 
the examples of Figures \ref{fig:tensegrity-examples} (on the right) and \ref{fig:general-position} also show that 
the PSD stress matrix of rank $n-d-1$ is not necessary for
dimensional rigidity.

In order to start to understand what \emph{is} necessary
for dimensional (and universal) rigidity we begin with
the following, Theorem 6 in
\cite{Alfakih-bar-frameworks}.
We also provide a simple proof as a special case of the results is Section
\ref{sect:Affine} here.

\begin{theorem}\label{thm:Alfakih-stress} 
If $(G,\p)$ is a dimensionally rigid framework with $n$ vertices whose
affine span is d dimensional, $d \le n-2$, then it has a non-zero equilibrium
stress with a PSD stress matrix $\Omega$.
\end{theorem} 

Note that the rank of $\Omega$ in Theorem \ref{thm:Alfakih-stress}
could be as low as one.
As such, it is weaker than the sufficient conditions above.
Later, we will describe a new condition, which is 
stronger than having a non-zero PSD stress matrix, but weaker
than having a non-zero PSD stress matrix of rank $n-d-1$. Our condition
instead
will be of the form of a sequence of PSD matrices, where the
combined rank is $n-d-1$.  Briefly, we will apply
Theorem \ref{thm:Alfakih-stress} repeatedly to a smaller and smaller
space of possible configurations.

\section{The measurement set}\label{sect:measurement}

Fix a finite graph $G$ with $n$ vertices, $m$ edges and fix a
Euclidean space $\R^D$, where the dimension $D$ is at least as large
as $n$. Let 
\[
\Cal:=\{\p \mid \pn \text{ is a configuration
  in } \R^D \}
  \]
be the set of configurations in $\R^D$. 
Each configuration can be regarded as a vector in $\R^{Dn}$. 

\begin{definition}We define the \emph{rigidity map} as
\begin{equation*}
f: \Cal =\R^{nD}  \rightarrow \CalM \subset \R^m
\end{equation*} 
by $f(\p) = (\dots, (\p_i -\p_j)^2, \dots)$, where the $\{i,j\}$ are
the corresponding edges in $G$, and $\CalM = \CalM(G)$ is the image of
$f$ in $\R^m$ for the graph $G$, which we call the \emph{measurement
  set}.  
\end{definition}
  
  In other words, $\CalM$ is the set of squared lengths of
edges of a framework that are actually achievable in some Euclidean
space.    There are some basic properties of $\Cal$ and any affine set $\CalA$ as below.

\begin{enumerate}
\item \label{cond:cone} $\CalM$ is a closed convex cone in $\R^m$.
\item \label{cond:equiv} For any $\e \in \CalM$, 
$f^{-1}(\e)$ consists of an equivalence class 
of frameworks $\p \in \Cal$.
\end{enumerate}
The convexity of Condition \ref{cond:cone} is well-known and even
has an explicit formula for the convexity in
\cite{Bezdek-Connelly-two-distance} and follows from Lemma \ref{lem:metric-universality} in Section \ref{sect:Affine}.  
Condition \ref{cond:equiv}
 follows directly from the definition.

\begin{definition}The \emph{rigidity matrix}  is defined as $R(\p)=\frac{1}{2}df_{\p}$, with respect to the standard basis in Euclidean space, and
$f(\p)=R(\p)\p$, where $df$ is the differential of $f$.  Then the energy function associated to a stress $\omega$ can also be written as 
\[E_{\omega}(\p) = \omega R(\p)\p,\]
where $\omega$ is regarded as a row vector.
\end{definition}

\section{Affine sets}\label{sect:Affine}

\begin{definition} A subset $\CalA \subset \Cal$ that is
the finite intersection of sets of the form
\begin{equation}\label{eqn:affine}
\{\p \in \Cal \mid \sum_{ij}\lambda_{ij}( \p_i -\p_j)=0 \}, 
\end{equation}
for some set $\{ \dots, \lambda_{ij}, \dots\}$ of constants, is called
an \emph{affine set}.
\end{definition}
Clearly
an affine set is a linear subspace of the configuration
space $\Cal$ and it is closed under arbitrary affine transformations
acting on $\R^D$. Moreover, any  
such set can be defined by equations of the form
(\ref{eqn:affine}).


For example, if there are three collinear points $\p_1, \p_2, \p_3$,
and $\p_2$ is the midpoint of $\p_1$ and $\p_3$, then $\{\p \in \Cal
\mid (\p_1-\p_2)-(\p_3-\p_2)=0\}$ is an affine set.  Or $\{\p \in \Cal
\mid \p_1-\p_2+ \p_3-\p_4=0\}$, which is a configuration of four
points of a parallelogram (possibly degenerate), is another example.

A special case of such an affine set is determined by a stress
$\omega$, where the equilibrium condition (\ref{eqn:equilibrium}) at
each vertex supplies the condition (\ref{eqn:affine}).

In Definition \ref{def:conic-infinity} we defined what it means for a configuration $\p$ to be universal with respect to a single stress $\omega$.  This just means that any other configuration $\q$ that is in equilibrium with respect to $\omega$ is an affine image of $\p$.  We generalize this the case to that of any affine set as follows.  

\begin{definition} We say that a configuration $\p$ in an affine set $\A$ is \emph{universal with respect to $\A$}, if any other configuration $\q$ in $\A$ is an affine image of $\p$.
We denote by $\mathring{\A} \subset \A$,  the set of configurations
that are universal with respect to $\A$.
\end{definition}

For any set $X$
in a linear space, $\langle X \rangle$ denotes the affine linear span
of $X$.

\begin{lemma}\label{lem:metric-universality} 
A configuration $\p \in \Cal$ is universal with respect to an affine set $\A$ if and only if it has maximal dimensional affine span for configurations in $\A$. Let $f:\A \rightarrow  \R^m$ be the restriction of the rigidity map to the measurement space for some graph $G$.  Then $f(\A)$ is convex and $f(\mathring{\A})$ is the relative interior of  
$f(\A) \subset \langle f(\A) \rangle$.  \end{lemma}
\begin{proof}  
Clearly any possible universal configuration must have maximal affine span in order for it to map affine linearly onto any other configuration in $\A$.  Conversely, let $\p$ be any configuration with maximal dimensional affine span, say $d$, in $\A$, and let $\q$ be any other configuration in $\A$.  Define $\tilde{\p}$ to be another configuration where $\tilde{\p}_i=(\p_i,\q_i) \in \R^D \times \R^D$ for $i=1, \dots, n$.  The configuration $\tilde{\p}$ is also in $\A$ since all its coordinates satisfy the equations (\ref{eqn:affine}).  Since projection is an affine linear map and the affine span of $\p$ is maximal, namely $d$, the dimension of the affine span of $\tilde{\p}$ must also be $d$, and the projection between their spans must be an isomorphism.  So the map $\p \rightarrow \tilde{\p} \rightarrow \q$ provides the required affine map since projection onto the other coordinates is an affine map as well.

If $\p,\q \in \A$, then, regarding $\p$ and $\q$ as being in complementary spaces, 
\begin{equation}\label{eqn:metric}
f((\cos \theta) \p, (\sin \theta) \q)=(\cos \theta)^2 f(\p)+(\sin\theta)^2  f(\q), 
\end{equation}
for $0 \le \theta \le \pi/2$ is the segment connecting $f(\p)$ to $f(\q)$ 
is in $f(\A)$ showing that $f(\A)$ and $f(\mathring{\A})$ are convex.

The rank of $df_{\p}$ is constant for non-singular affine images of  $\p$ (see \cite{Connelly-Whiteley}, for example), which are in $\mathring{\A}$, the universal configurations. This implies that $f$ is locally a projection into $f(\A)$ at $\p$, which  implies that 
$f(\mathring{\A})$ is 
open in 
$\langle f(\A) \rangle$.   This, combined with 
its being dense in $f(\A)$, 
and its convexity
makes 
$f(\mathring{\A})$ equal to 
the relative interior of 
$f(\A)$. $\Box$
\end{proof}

The dimension of an affine set $\A$ is $\dim(\A)=D(d+1)$,
where $D$ is the dimension of the ambient space and $d$ is the
dimension of the affine span of a universal configuration $\p$ for
$\A$.

For any (symmetric) bilinear form $B$ for a vector space $V$, the
\emph{radical of $B$} is the set $\{\v \mid B(\v, \w) = 0\,\,
\text{for all}\,\, \w \in V\}$.  If $V$ is a finite dimensional vector
space and $B$ is given by a symmetric matrix, then the radical of $B$
is the kernel (or co-kernel) of that matrix.  We can interpret the
stress-energy $E_{\omega}$ as such a bilinear form.  
If  $B$ acting on  $V$ is PSD, then its zero set must
be equal to its radical.

\begin{lemma}\label{lem:stress-boundary} 
Let $\q \in \CalA \subset \R^{nD}$.  Then $f(\q)$ is in the boundary
of the relative interior of $f(\CalA) \subset \langle f(\CalA)
\rangle$ if and only if there is a non-zero stress $\omega$ for $(G,
\q)$ such that when
$E_{\omega}$ is restricted to $\CalA$, the resulting form is PSD
and has 
$f(\q)$ in its radical.
\end{lemma}
Note that this does NOT mean that the
$E_{\omega}$ 
is necessarily PSD over all of $\Cal$ 
or that the
configuration $\q$ is in the radical of 
the form $E_{\omega}$ 
defined over all of $\Cal$.

\begin{proof} 
Suppose that a stress $\omega \ne 0$ exists for the framework $(G,\q)$.
The condition that $E_{\omega}$ is PSD on $\A$ is equivalent to
$E_{\omega}(\q) \ge 0$ for all $\q$ in $\A$, which is equivalent to
the linear inequality $\omega f(\q) \geq 0$ for any $f(\q) \in f(\A$),
and any configuration $\q$ in $\CalA$.  When $E_{\omega}(\q)=0$, then
$f(\q)$ is in the closure of the complement of that inequality in
$\langle f(\CalA) \rangle$ and thus in the boundary of $f(\CalA)
\subset \langle f(\CalA) \rangle$.  

Conversely, suppose that $f(\q)$ is in the boundary of $f(\CalA)
\subset \langle f(\CalA) \rangle$.  Since the set $f(\CalA)$ is
convex, $f(\q)$ is in a supporting hyperplane
\[
\CalH = \{ \e \in \langle f(\CalA) \rangle \mid \omega   \e =0\},
\]
which is defined by a non-zero stress $\omega$.  Then
\begin{equation*}
0 \le \frac{1}{2}\omega   f(\q) = \omega R(\q)\q=\sum_{i<j}\omega_{ij}(\q_i-\q_j)^2=\q^t\Omega \otimes I^D \q = E_{\omega}(\q).
\end{equation*}

Thus 
the quadratic form defined by $E_{\omega}$ restricted to the affine set $\CalA$, 
is PSD
and has
$f(\q)$  in its radical.
\end{proof}

\begin{lemma} 
\label{lem:radAff}
Let $\A$ be  an affine set
and $E_\omega$ be a stress energy which we 
restrict to  $\A$. 
Then its  radical
must be an affine set. 
\end{lemma}
\begin{proof}
Let $\q$ be universal for $\A$. 
Then a configuration $\p \in \A$ is in the radical when 
\[
\sum_{i<j} \omega_{ij} (\p_i-\p_j)\cdot(\tilde{\q}_i-\tilde{\q}_j) =0,
\]
for any $\tilde{\q}$ that is an affine image of $\q$.

Suppose some $\p$ is in the radical. Then clearly so is any translation of $\p$.
Any linear transform applied to the coordinates of $\p$ can be defined using the above
equation by applying its inverse transpose to $\q$.
Thus the radical is invariant for affine transforms, 
making it an affine set.
$\Box$ \end{proof}

\section{Iterated affine sets and the main theorem}\label{sect:affine-constraints}

\begin{definition} If  
$\Cal = \A_0 \supset \A_1 \supset \A_2 \supset \dots \A_k$ is a sequence of affine sets,  we call it
 an \emph{iterated affine set}. 
 \end{definition}

\begin{definition} 
Suppose an iterated affine set has a corresponding  
 sequence of stress energy
functions $E_1, \dots, E_k$ as defined 
of the form (\ref{eqn:stress-energy}) such that
 each $E_i$ is restricted to act only on $\A_{i-1}$.
Suppose that each  restricted $E_i$
is PSD (over $\A_{i-1}$), that $E_i(\q)=0$ for all $\q \in \A_i$, and that
$E_i(\q) > 0$ for all $\q \in  \A_{i-1} - \A_i$.
Then we call $E_1, \dots, E_k$ 
an \emph{(associated) iterated PSD stress} for this iterated affine set. 
 \end{definition}

Our main result is the following
characterization of dimensional rigidity.

\begin{theorem}\label{thm:main} 
Let $(G,\p)$ be a framework in $\R^d$, where $\p$ has an affine span of dimension $d$. Suppose
 $\Cal = \A_0 \supset \A_1 \supset \A_2 \supset \dots \A_k$
is an iterated affine set with $\p \in \A_k$ and with an associated iterated PSD stress. 
If the dimension of $\A_k$ is $(d+1)D$.
Then $(G,\p)$ is dimensionally  rigid
in $\R^d$. 

Conversely, if $(G,\p)$ is dimensionally rigid in $\R^d$, then there must
be an iterated affine set with $\p \in \A_k$, $\Dim(\A_k) = (d+1)D$,
with an associated iterated PSD stress.
\end{theorem}

\begin{proof}  
First we prove the easy direction.
Since $E_i$ operates on the squared edge lengths,  the energy function
forces any equivalent framework ($G,\q)$ to be in $\A_{i}$ and ultimately
in $\A_k$.  Since the dimension of $\A_k$ is $(d+1)D$, $\p$ must be universal for $\A_k$, and so $\q$ must be
an affine image of $\p$ and thus has, at most, a $d$-dimensional affine span.

For the converse, suppose that $(G,\p)$ is dimensionally rigid in
$\R^d$.  The configuration $\p$ is such that $\p \in \Cal =\A_0$. If
$f(\p)$ is in the boundary of $f(\A_0)$ we apply Lemma
\ref{lem:stress-boundary} to find a stress $\omega_1$ and a
corresponding stress-energy function $E_{1}$ whose radical 
includes $\p$, and by Lemma~\ref{lem:radAff}
is an 
affine set $\A_1$
In order to iterate the process we define
\begin{eqnarray}\label{eqn:Ai}
\A_i= \{\q \in \A_{i-1} \mid \omega_iR(\q)\q=0\},
\end{eqnarray}
where $\omega_i \ne 0$ is chosen such that $\omega_i R(\q)\q= \omega_i
f(\q) \ge 0$, for all $\q \in \A_{i-1}$, $\omega_i R(\q)\q = \omega_i
f(\q) > 0$ for some $\q \in \A_{i-1}$, and $\omega_i R(\p)\p= \omega_i
f(\p) = 0$.  The quadratic form $\q^t \Omega_i \otimes I^D \q$ is PSD
when restricted to $\A_{i-1}$, and 
from Lemma~\ref{lem:radAff},
the resulting $\A_i$ must also be an affine set.
When
such an $\omega_i \ne 0$ cannot be found, we stop and that is the end
of the sequence of affine sets.  This sequence must terminate as each
of our subsequent affine sets is of strictly lower dimension.

From Lemma \ref{lem:stress-boundary} we see that we can continue creating stresses
$\omega_1, \dots, \omega_k$ and affine sets until $f(\p)$ is in the
relative interior of $f(\A_k)$, and is universal with respect to
$\A_k$ by Lemma \ref{lem:metric-universality}.  If the dimension of
$\A_k$ is not $D(d+1)$, then the dimension of $\A_k$ is strictly
greater than $D(d+1)$ and the dimension of the affine span of $\p$
would have been greater than $D(d+1)$, a contradiction. $\Box$
\end{proof}

 Figure \ref{fig:race-track}, similar to Figure 2 of \cite{Gortler-Thurston2}, shows a symbolic version of this process in the measurement set, where the indicated point represents the image of the configuration and its relation to the measurement cone.  The arrows represent the stress vectors.
\begin{figure}[here]
    \begin{center}
        \includegraphics[width=0.5\textwidth]{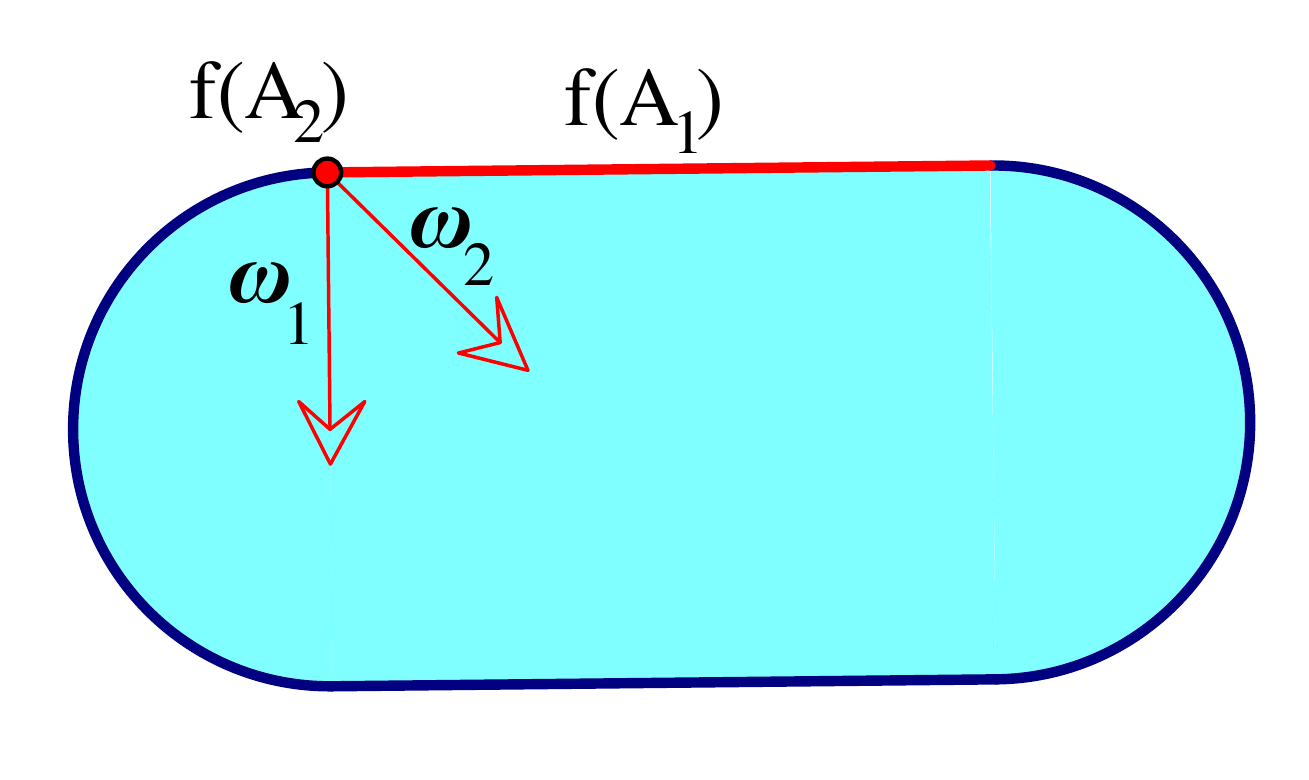}%
        \end{center}
    \caption{The sets $f(\A_1)$ and $f(\A_2)$ shown as the point and line segment.}
    \label{fig:race-track}
    \end{figure}

\subsection{The basis matrix}\label{subsect:basis-matrix}

An affine set $\A$ can always be represented by a universal 
configuration $\b = (\b_1 \dots \b_n)$ of $n$ points in $\R^D$, with an affine
span of some dimension, say $d$.
Without loss of generality, we
can assume (using a translation if needed)
that the linear span of the $\b_i$ (thought of as vectors) is of dimension $d+1$.

\begin{definition}We  define a \emph{basis matrix} $B$, for an affine set
as a rank $d+1$ matrix with 
$n$ columns and $D$ rows given by the coordinates of $\b$.
\end{definition}

Since this matrix has rank $d+1$, we can then apply row reduction operations so that $B$ has only $d+1$ rows. Additionally, (if we want) since the affine span of $\b$ is only $d$ dimensional, we can perform these operations so that the final row is the all-ones vector.

\begin{definition}
Given an iterated affine set,
$\Cal = \A_0 \supset \A_1 \supset \A_2 \supset \dots \A_k$.
We define $d_i$ to be the dimension of the affine span
of a universal configuration for $\A_i$. 
\end{definition}

\begin{definition}
Given an iterated PSD equilibrium stress for an iterated affine set, a
basis matrix $B_{i-1}$ for each $\A_{i-1}$, and the $n$-by-$n$ stress
matrix $\Omega_i$ corresponding to each $E_i$, we define 
a \emph{restricted stress matrix}
$\Omega^*_i
:= B_{i-1} \Omega_i B^t_{i-1}$.  Each $\Omega^*_i$ is a
$(d_{i-1}+1)$-by-$(d_{i-1}+1)$ PSD matrix.
\end{definition}

The following is a Corollary of Theorem \ref{thm:main}.

\begin{corollary}\label{cor:mainD} 
Let $(G,\p)$ be a framework in $\R^d$ with an affine span of dimension
$d$.  Suppose $\Cal = \A_0 \supset \A_1 \supset \A_2 \supset \dots
\A_k$ is an iterated affine set with $\p \in \A_k$, and that this
iterated affine set has an associated iterated PSD stress, described
by restricted stress
matrices $\Omega^*_i$.  Let $r_i$ be the rank of $\Omega_i^*$.  If
\begin{equation}\label{eqn:affine-rank}
\sum_{i=1}^{k} r_i = n-d-1.
\end{equation}
then $(G,\p)$ is dimensionally rigid.
Conversely, if $(G,\p)$ is dimensionally rigid in $\R^d$, then there must
be an iterated affine set with $\p \in \A_k$,
with an associated iterated PSD stress such that 
Equation~(\ref{eqn:affine-rank}) holds.
\end{corollary}

The two versions of this theorem are related as follows:
the zero set
of configurations for the energy function $E_i$ corresponds via the
change of basis $B_{i-1}$,  to the kernel of
the matrix $\Omega^*_i$.  Since the rank of $\Omega^*_i$ is
$r_i$, its kernel has dimension $d_{i-1} +1-r_i=d_i+1$.
Thus $d_k+1 = n -\sum_{i=1}^i r_i =(d+1)$.  So
$\A_k$, which has dimension $(d_k+1)D$, is the set of all affine images of $\p$ in
$\R^D$.  

Figure \ref{fig:ladder}, described later, is an example of an application of Theorem
\ref{thm:main}.  The set of configurations of all the points, where
for a pole, one is at the midpoint between the other two, is an affine
set.  The stress is indicated.  Each of the restricted stress matrices
has rank one.  The horizontal members also have a stress that is in
equilibrium when restricted to the intersection of the first two
affine sets.  This matrix also has rank one.  Thus all the stress
matrices can be assumed to be (and are) PSD.  But $n=6, d=2$, so $d+1
+\sum_{i=1}^i r_i=3+ 3 = 6 = n$, and this $(G,\p)$ is dimensionally
rigid in $\R^2$.  This framework
has a flex in the plane that is an affine motion, but the point is
that it cannot be twisted into a $3$-dimensional shape.  The
calculations are done in Subsection \ref{subsect:ladder}.

One application of Theorem \ref{thm:main} is to universal rigidity.

\begin{corollary}\label{cor:main} 
Suppose $\Cal = \A_0 \supset \A_1 \supset \A_2 \supset \dots \A_k$ is an 
iterated affine set for a framework $(G,\p)$ with $n$ vertices in
$\R^d$. Suppose that the iterated affine set has an associated
iterated PSD stress. If $\dim(\A_k)=D(d+1)$ 
and the member directions do not lie on a conic at infinity, then $(G,\p)$
is universally rigid. 

Conversely if $(G,\p)$ is universally rigid in
$\R^d$, then there is such an iterated affine set 
with an iterated PSD stress,
and the member directions do not lie on a
conic at infinity.
\end{corollary}

For example, if another bar is inserted between any of two of the
vertices that do not already have a bar in Figure \ref{fig:ladder},
the resulting framework will be universally rigid.

\section{Convexity interpretation}\label{sect:convexity}

We now point out the connection of the results here from the
point of view of basic convexity considerations.  

\begin{definition} For any finite
dimensional convex set $X$ and any point $x$ in $X$, let $F(x)$, 
called the \emph{face} of $x$, be
the largest convex subset of $X$ containing $x$ in its relative
interior.  
Equivalently~\cite{Grunbaum-polytopes}, $F(x)$ is the set of points $z \in X$ so that there is a
$z' \in X$ with $x$ in the relative interior of the segment $[z',z]$.
\end{definition}

\begin{definition} A subset $X_0 \subset X$ is called a
\emph{face} of $X$ if $X_0 = F(x)$ for some $x \in X$.     
\end{definition}

\begin{definition}

Let $X= X_0 \supset X_1 \supset X_2 \supset \dots X_k$ be a
sequence of faces of $X$, which we call a \emph{face flag}.  If
each $X_i = \CalH_i \cap X_{i-1}$, where $\CalH_i \subset \langle
X_{i-1}\rangle$ is a support hyperplane for $X_{i-1}\subset
\langle X_{i-1}\rangle$ for $i=1, \dots, k$, then we call the
face  flag \emph{supported}.  
\end{definition}
The
following is an easy consequence of these definitions. 

\begin{lemma} A subset $Y$ of $X$ is a face of $X$ if and only if
$Y = X_k$, for some supported flag face.  \end{lemma}

We next specialize to the case when the space $X=\CalM$, the
measurement space for the graph $G$ defined in Section
\ref{sect:measurement}.  The function $f$ is the rigidity map as
before.

\begin{lemma}\label{lem:hyperplane-convexity} A supporting
hyperplane $\CalH \subset \R^m$ for $\CalM$ corresponds to a
non-zero PSD stress $\omega$ for the graph $G$.  A hyperplane $\CalH$
supports a convex subcone of $X_i \subset \CalM$ if and only if
there is a quadratic energy form $E_{\omega}$ which is PSD on
$f^{-1}(X_i)$ and $E_{\omega}(\p)=0$ for some $\p \in
f^{-1}(X_i)$.  \end{lemma}

\begin{definition} If $\p \in \Cal$ is a configuration, define $\A(\p)$ to be the
set of all affine images of $\p$. 
As before,  we call any $\q$ of maximal
dimensional affine span in $\A(\p)$ a \emph{universal
configuration} for  $\A(\p)$.  Define
$\mathring{\A}(\p)$ to be the set of universal configurations of
$\A(\p)$.
\end{definition}

\begin{lemma}\label{lem:faceInAffine} 
Suppose  $(G,\p)$ is
dimensionally rigid.
Then
\[f^{-1}(F(f(\p))) \subset \A(\p).\]
Thus additionally, we have
\[F(f(\p)) \subset f(\A(\p)).\]
\end{lemma}

\begin{proof}
Suppose not. 
Then 
there is a configuration $\q \not\in \A(\p))$
but such that $f(\q) \in F(f(\p))$. 
Since $f(\p)$ is in the interior of
the face, and f(\q) is in the face, then, from the definition of a face,
there must be some third
configuration $\r$, such that $f(\p)$ is in the relative interior
of the segment [f(\q),~f(\r)]. As in the proof of Lemma
\ref{lem:metric-universality}, we can use 2 complementary spaces,
and find appropriate scalars $\alpha$ and $\beta$ such that
$\tilde{\p} := (\alpha \q, \beta \r)$ is equivalent to $\p$. But
since $\q$ is not an affine image of $\p$, then neither is
$\tilde{\p}$. This, together with Theorem \ref{thm:affineness},
contradicts our assumption that $\p$ was dimensionally
rigid. 
$\Box$ \end{proof}

\begin{lemma}\label{lem:AffineInFace} 
Suppose  $(G,\p)$ is
dimensionally rigid.
Then
\[F(f(\p)) \supset f(\A(\p)).\]
Thus additionally, we have
\[f^{-1}(F(f(\p))) \supset \A(\p).\]
\end{lemma}
\begin{proof}
From Lemma~\ref{lem:metric-universality},
we know that  $f(\A(\p))$ is convex with $f(\p)$ in its
relative interior.
Thus from the definition of a face, we have
$F(f(\p))) \supset f(\A(\p))$.
$\Box$ \end{proof}

\begin{corollary}
If $(G,\p)$ is dimensionally rigid and the configuration $\q$ is
a non-singular affine image of $\p$, then $(G,\q)$ is
dimensionally rigid as well.
\end{corollary}

\begin{proof}
Since $\q \in \A(\p)$, then from the above lemmas, we have
$f^{-1}(f(\q)) \in \A(\p)$. But  $\q$ is universal for $\A(\p)$ and so 
$\A(\p)=\A(\q)$, thus making $\q$ dimensionally rigid.
\end{proof}

With the above two Lemmas in mind, we make the following definition:

\begin{definition}We say that the affine set $\A$ is a \emph{$G$-affine set} if
 $\A$  is equal to the pre-image of some face of the measurement set.
\end{definition}

\begin{proposition}\label{prop:main}
A framework $(G,\p)$ is
dimensionally rigid if and only if $\A(\p)$ is a $G$-affine set.
\end{proposition}
\begin{proof}
Suppose that $(G,\p)$ is dimensionally rigid.
Then from Lemmas
\ref{lem:faceInAffine} and \ref{lem:AffineInFace}, 
we know $f^{-1}(F(f(\p)))=\A(\p)$, which
is thus a $G$-affine set.

For the other direction, let $F'$ be any face of $\CalM$
containing $f(\p)$.  Then $F(f(\p)) \subset F'$.  If $(G,\p)$ is
not dimensionally rigid, then there is configuration $\q \not\in
\A(\p)$ such that such that $f(\q)=f(\p)$.  Thus $f^{-1}(F(\p))$
is not a subset of $\A(\p)$, and $f^{-1}(F')$ is not a subset of
$\A(\p)$.  So $\A(\p)$ is not a $G$-affine set.
$\Box$ \end{proof}

In summary this says that the face lattice of the measurement set
$\CalM$ exactly corresponds the lattice of $G$-affine sets.
Theorem \ref{thm:main} follows directly. The sequence of faces in
a face flag of $\CalM$ corresponds to an iterated sequence of
$G$-affine sets $\A_i$ cut out by an appropriate stress sequence
$E_i$.

\section{Relation to Facial Reduction}\label{sect:facial}

Facial reduction is a general technique used in the study of duality
in cone programming~\cite{Borwein-Wolkowicz, Ramana, Pataki}, and here we describe the
translation between that and our exposition here.
In the general setup,
one might have a cone programming problem where the feasible set
is expressed as points $\x \in \R^N$ that are
both in  some convex cone $K \subset \R^N$ and satisfy
an equality constraint, expressed as $\x \in L + \b$, where $L$ is a linear
subpace of  $\R^N$
and $\b \in \R^N$. Let $\x_0$ be in the relative interior of the feasible
set and let $F_{\min} := F(\x_0)$ be its face in $K$.

In the process of facial reduction, we start with $F_0:=K$ and
find a supporting hyperplane $\Omega_1^\perp$ 
whose intersection with  $F_0$ 
is
some subface $F_1$ of $F_0$ such that $  F_1 \supset F_{\min}$.
This can be iterated on any $F_{i-1}$ by finding a hyperplane
$\Omega_i^\perp$ that
supports $F_{i-1}$ and whose intersection with
$F_{i-1}$ is some subface
$F_{i}$
such that $ F_i \supset F_{\min}$.
In each step, we guarantee that we are not excluding
 any part of $F_{\min}$ by ensuring that
$\Omega_i \in (L^\perp \cap \b^\perp)$.
This process is iterated until
$F_i = F_{\min}$.

In the setting of graph embedding, we can think of $K$ as
$S^n_+$,
the cone of n-by-n symmetric PSD matrices.
Any configuration $\p$ can be mapped to its Gram matrix in $K$.
Each affine set $\A$ corresponds to some face of
$S^n_+$. (Note that not every face $F$ of $S^n_+$ corresponds to
an affine set. The face $F$ must include the all-ones matrix so that
its corresponding configuration set is closed under translations in
$\R^D$).


The linear constraint $\x \in L+\b$
corresponds to a framework being equivalent to $(G,\p)$.
(The graph $G$ determines the space $L$ and the edge lengths in
$\p$ gives us a $\b$).
The constraint
$\Omega_i \in L^\perp$ 
means that $\Omega_i$ is a stress matrix for $G$ (zero on non edges,
and rows summing to zero).
The constraint
$\Omega_i \in \b^\perp$ 
means that any
$\p$ and any equivalent configuration has zero energy
under the quadratic form defined by $\Omega_i$.
The constraint that $\Omega_i$ supports $F_{i-1}$ corresponds to
$\Omega_i$ being PSD over a corresponding affine set $\A_i$.

Under this correspondence, one can see that our process of
finding iterated affine sets $\A_i$ using iterated stress matrices
$\Omega_i$ corresponds exactly to an application of facial reduction.

We note, that in our exposition, we do not describe the process using
$S^n_+$
at all.
On the one hand,
we describe the affine sets $\A_i$ as subsets
of configuration space (instead of as  faces of $S^n_+$).
On the other hand,
instead of picturing of our stresses $\Omega_i$ as support planes for
$S^n_+$
we work over the measurement set of our graph
$\CalM(G) := S^n_+/L$, which is a linear projection of $S^n_+$.
In this projected picture,
our support planes are orthogonal to the stress vectors $\omega_i$ in $\R^m$.

As described in Section \ref{sect:convexity}, facial reduction ``upstairs" on the cone $K$
(such as $S^n_+$) for the constraint $x \in L+\b$
is exactly mirrored by the facial reduction
``downstairs" on the cone $K/L$ (such as $\CalM$) for the constraint
$x = \b/L$.

\section{Tensegrities}\label{sect:tensegrities}

It is also possible to use the ideas here to get a similar complete
characterization of universal rigidity for tensegrity frameworks,
where there are upper and lower bounds (cables and struts) on the member
lengths corresponding to the sign of the rigidifying stresses.  

\begin{definition}Each edge of a graph $G$ is designated as either a \emph{cable}, which is constrained to not get longer in length, or a strut, which is constrained not to get shorter in length, or a bar, which, as before, is constrained to stay the same length.  So when we have a framework $(G,\p)$, where each edge, which we call a \emph{member}, is so designated, we call it a \emph{tensegrity framework}, or simply a \emph{tensegrity}, and we call $G$ a \emph{tensegrity graph}. \end{definition}

We can then ask whether $(G,\p)$ is locally rigid, globally rigid, or universally rigid.  
For local rigidity and the corresponding concept of infinitesimal rigidity, there is an extensive theory as one can see in \cite{Connelly-What-is,Shaping-space, Roth-Whiteley, Polyhedral-tensegrity, Skelton, Recski, Connelly-Whiteley}, for example.  For global rigidity and universal rigidity, there is a natural emphasis on stress matrices and related ideas. 

\begin{definition}\label{def:tensegrity}
We say that a stress $\omega=(\dots, \omega_{ij}, \dots)$ for a tensegrity graph is a \emph{proper} stress if $\omega_{ij} \ge 0$, when the member $\{i,j\}$ is cable, and $\omega_{ij} \le 0$, when the member $\{i,j\}$ is a strut.  There is no condition for a bar.  
\end{definition}

Theorem \ref{thm:fundamental} takes on the following form for
tensegrities.  See \cite{Connelly-energy}.
 
 \begin{theorem} \label{thm:fundamental-tensegrity}  Let $(G,\p)$ be a tensegrity framework whose affine span of $\p$ is all of $\R^d$, with a proper equilibrium stress $\omega$ and stress matrix $\Omega$.  Suppose further
\begin{enumerate}
\item \label{condition:positive} $\Omega$ is PSD.
\item \label{condition:rank} The configuration $\p$ is universal with respect to the stress $\omega$.  (In other words, the rank of $\Omega$ is $n - d - 1$.)
\item \label{condition:conic} The member directions of $(G,\p)$ with a non-zero stress, and bars, do not lie on a conic at infinity.
\end{enumerate}
Then $(G,\p)$ is universally rigid.
\end{theorem}

When we draw a tensegrity, cables are designated by dashed line segments, struts by solid line segments, and bars by thicker line segments, as in Figure \ref{fig:tensegrity-examples}.
\begin{figure}[here]
    \begin{center}
        \includegraphics[width=0.8\textwidth]{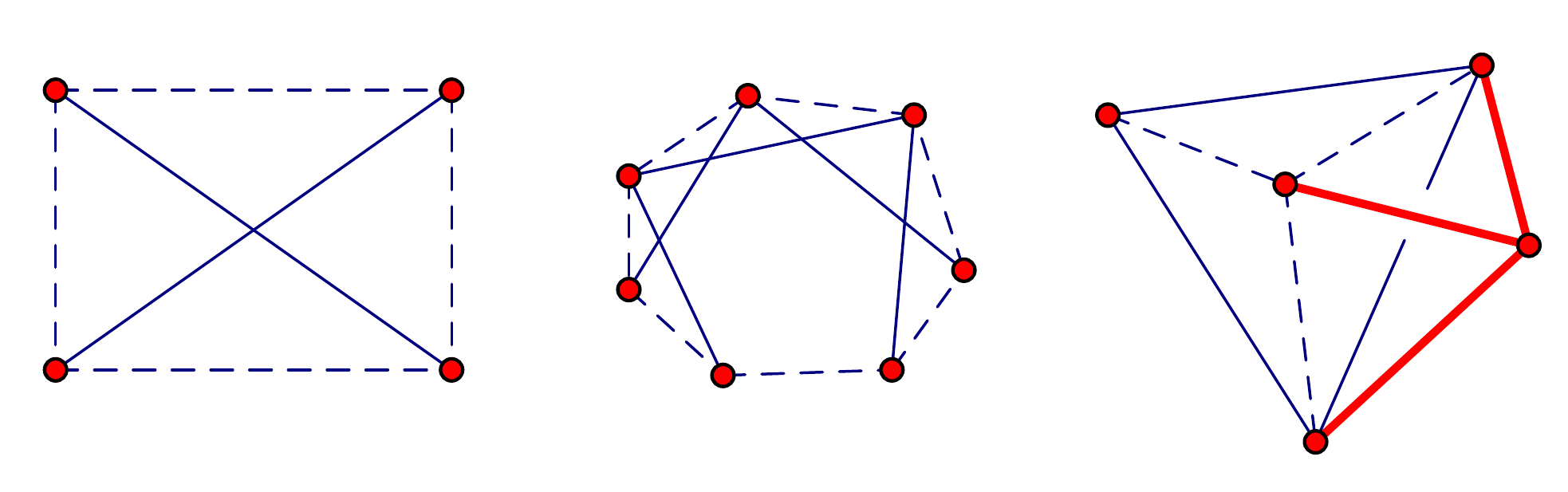}%
        \end{center}
    \caption{These are three examples of super stable tensegrities.  The one on the left is trivially universally rigid when all the members are bars.  But as a tensegrity it is also super stable, which follows from its rank one equilibrium stress matrix.  The tensegrity in the middle is an example of a Cauchy polygon, one of the class of convex polygonal tensegrity polygons as defined in \cite{Connelly-energy}.  The one on the right has a degree three vertex attached by bars to another super stable planar tensegrity in $\R^3$.  The bars must have zero stress, but in order to insure that there is no affine motion, the bar directions must be included in the directions that are to avoid the conic at infinity.}
    \label{fig:tensegrity-examples}
    \end{figure}

\section{Iterated stresses for tensegrities}\label{sect:iterated-tensegrities}

For the case of tensegrities, the iterated case is similar.  

\begin{definition}We say that a tensegrity $(G,\p)$ in $\R^d$ is \emph{dimensionally rigid}, if any other configuration $\q$ in any $\R^D$, satisfying the member constraints of $G$ has an affine span of dimension $d$ of less. 
\end{definition}
 
 \begin{theorem}\label{thm:main-tensegrity}
 Let $(G,\p)$ be a tensegrity in $\R^d$, where $\p$ has an affine span of dimension $d$. Suppose
 $\Cal = \A_0 \supset \A_1 \supset \A_2 \supset \dots \A_k$
is an iterated affine set with $\p \in \A_k$ with an associated iterated proper PSD stress. 
If the dimension of $\A_k$ is $(d+1)D$,
then $\p$ is dimensionally  rigid
in $\R^d$. 

Conversely, if $(G,\p)$ is dimensionally rigid in $\R^d$, then there must
be an iterated affine set with $\p \in \A_k$, $\Dim(\A_k) = (d+1)D$,
with an associated iterated proper PSD stress.

\end{theorem}

\begin{proof} The proof of this is essentially the same as in Section \ref{sect:affine-constraints} for
Theorem \ref{thm:main}.
For the necessity direction, we just need to be careful
to maintain the proper signs for a tensegrity stress.
When a tensegrity is dimensionally rigid, this means that not only
is $f(\p)$ on the boundary of $\CalM$,
but that $\CalP$, the polyhedral cone of
tensegrity constraints of Definition \ref{def:tensegrity} (the squared lengths $e^2_{ij} \le (\p_i-\p_j)^2$, for each cable, and $e^2_{ij} \ge (\p_i-\p_j)^2$, for each strut), is disjoint from $\CalM$, 
except for 
$f(\A(\p))$.
By a standard separation theorem,
we can choose
a hyperplane that  separates the relative interiors of the two convex sets
$\CalP$ and $\CalM$. (See Figure \ref{fig:cone-intersect} in the next section.) This means that the corresponding stress will
be a proper stress for the tensegrity.
It may be the case, that
this hyperplane
contains other points of the boundary of $\CalP$ besides just $f(\p)$,
which means that some of the edges of $G$ will have zero stress
components.
 This argument can be applied at each level of
iteration.

Note once an edge has a non-zero stress component at some level,
this strictness can be maintained
at any subsequent level.
In particular, the stress $\omega_i$ is orthogonal to
the configurations in $f(\A_j)$, thus we can always replace
$\omega_j$, for $j > i$,
with $\omega_i + \omega_j$. So once a member gets stressed, it can remain
stressed from then on.$\Box$
\end{proof}

\begin{figure}[here]
    \begin{center}
        \includegraphics[width=0.5\textwidth]{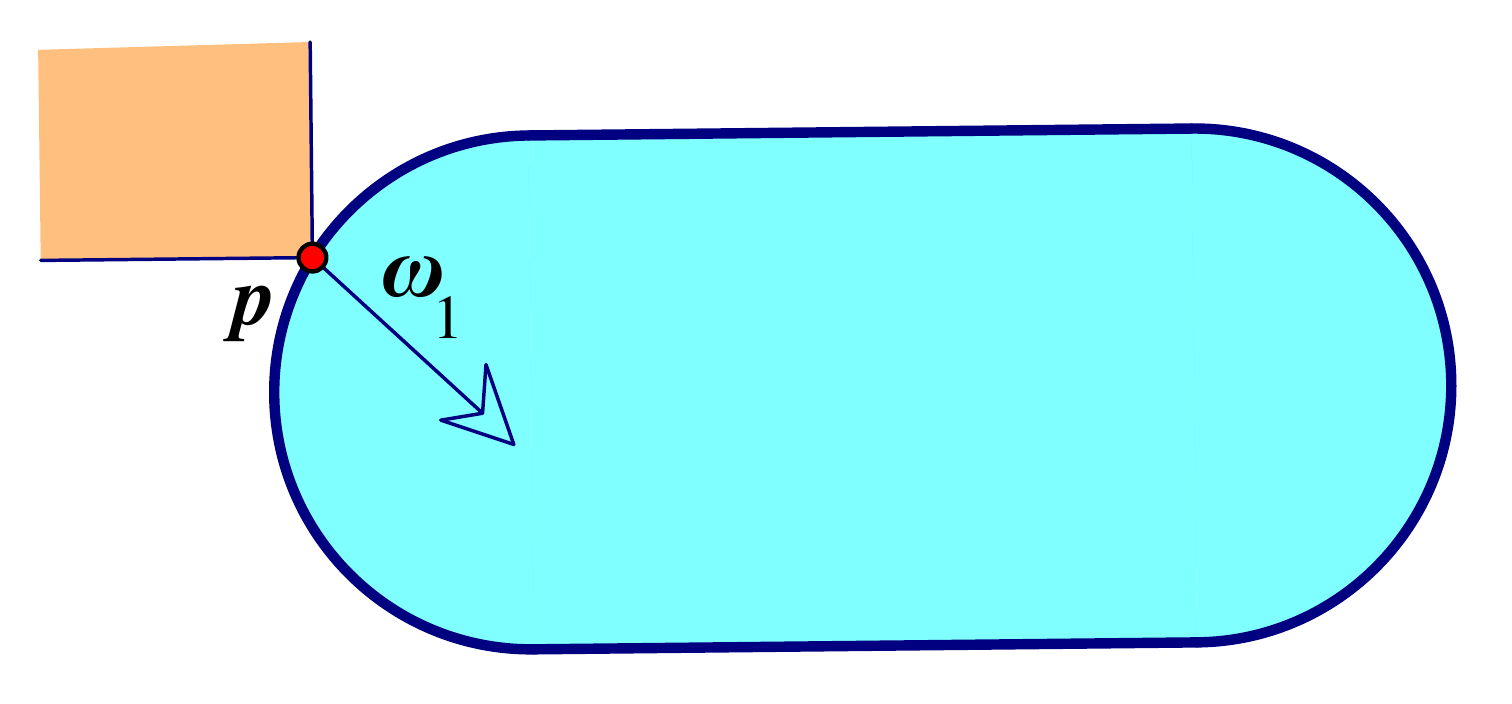}%
        \end{center}
    \caption{This shows a section of a cone as in Figure \ref{fig:race-track}, but with the rectangular cone given by the cable and strut constraints.  The stress vectors $\omega_1$ determine  the rectangular cone since it is proper.} 
    \label{fig:cone-intersect}
    \end{figure}
    
The major application of this result is the following.

\begin{corollary}\label{cor:main-tensegrity} 
Suppose $\Cal = \A_0 \supset \A_1 \supset \A_2 \supset \dots \A_k$ is
an iterated affine set for a tensegrity $(G,\p)$ with $n$ vertices in
$\R^d$, with an associated iterated proper PSD stress described by 
PSD restricted stress matrices  $\Omega^*_i$.
Let $r_i$ be the rank of 
$\Omega^*_i$. If  (\ref{eqn:affine-rank}) holds, and the
member directions with non-zero stress directions and bars do not lie
on a conic at infinity,   then $(G,\p)$ is universally rigid.

Conversely if $(G,\p)$ is universally rigid in $\R^d$, then there is
an iterated  affine set with an associated
iterated PSD stress determined by proper stresses,
the dimension of $\A_k$ is $(d+1)D$, and the members with non-zero
stress directions and bars do not lie on a conic at infinity.
\end{corollary}
\begin{proof}
This proof also  follows that of the case of a bar framework.
The only thing new that we need to establish in the
necessity direction is that we will be able to find
\emph{non-zero}
stress values on the
cable and strut edges to certify that they do not lie on a conic at infinity.
The iterated stresses that are guaranteed from the above theorem
need not be non-zero on any particular set of edges (See the example
of Figure \ref{fig:one-pole} below).

To establish this we can use, if needed, one extra stress
beyond that needed to establish dimensional rigidity.
Suppose at the last level of iteration, we have
a sequence of  stresses that
restricts us to frameworks in the affine set $\A_k$, such that
$\p$ is universal for $\A_k$. In this case, we have that
$f(\p)$ is in the relative
interior of $f(\A_k)$.
The assumption of universal rigidity means that
the polyhedral cone $\CalP$ 
is disjoint from $f(\A_k)$, except for the
shared point
$f(\p)$. Since
$f(\p)$ is in the relative
interior of $f(\A_k)$, this means that we can find a hyperplane
that includes $f(\A_k)$ and excludes \emph{all} of $\CalP$ except
for the single point $f(\p)$. The
corresponding stress must have zero energy
for all of $\A_k$ and will have non-zero values on
\emph{all} of the edges. $\Box$
\end{proof}
\begin{figure}[here]
    \begin{center}
        \includegraphics[width=0.5\textwidth]{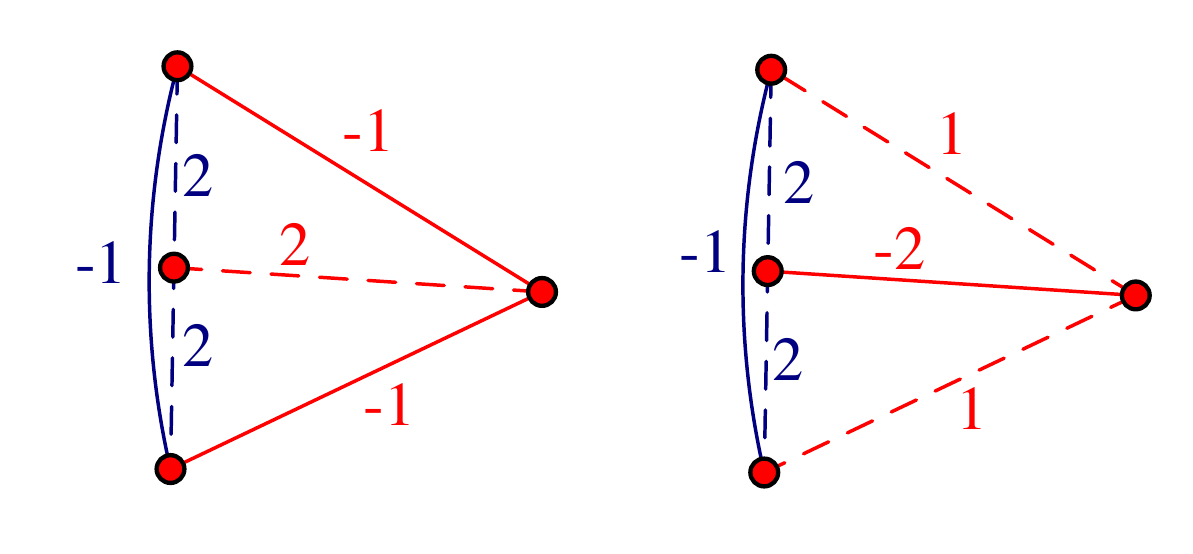}%
        \end{center}
    \caption{These are two universally rigid tensegrities in the plane.  The signs on the members not on the pole can be reversed and it still remains universally rigid.} 
    \label{fig:one-pole}
    \end{figure}

Figure \ref{fig:one-pole} is an example where one extra
iteration is needed
for universal rigidity 
after the iteration process shows dimensional
rigidity.  There is just one pole, in the plane, and just one vertex
attached to all three vertices.  There are two ways (as shown) to
assign cables and struts to the remaining three members so that there
will be an equilibrium at that vertex.  Both possibilities provide a
universally rigid tensegrity.  At the first level, we can find a rank
$1$ stress on the vertical pole.  This is sufficient to serve as a
certificate for dimensional rigidity. For a bar framework, universal
rigidity follows since the edge directions do not lie at a conic at
infinity.  But for a tensgrity framework, we are not done, since in
that case, the conic test only can use cable and strut edges with
non-zero stress coefficients. As shown in Figure \ref{fig:one-pole},
for this we can use a second level stress that has a constant $0$
energy over $\A_1$.

\section{Projective invariance}\label{sect:projective}

It is well known that a bar framework $(G,\p)$ is infinitesimally rigid if and only if  
$(G,\q)$ is infinitesimally rigid, where the configuration $\q$ is a non-singular projective image of the configuration $\p$.  See \cite{Connelly-Whiteley, Whiteley-Polarity-I,  Whiteley-Polarity-II} for a discussion of this property.  Infinitesimal rigidity for tensegrities is also projectively invariant, but a cable that ``crosses" the hyperplane at infinity is changed to a strut and vice-versa, because  the sign of the stress changes.  It is also true that any equilibrium stress is also altered by the projective transformation.  Indeed a stress matrix $\Omega$ is replaced by another stress matrix $D\Omega D$, where the matrix $D$ is a non-singular diagonal matrix and comes from the non-singular projective transformation. This transformation preserves the rank and PSD nature  of the stress.  At any subsequent level, we also can set $\Omega_i := D \Omega_i D$ using the same $D$ matrix. The basis matrix, which derives from the kernel is transformed as $B_i \rightarrow B_i D^{-1}$. Thus, the restricted stress matrix, 
$\Omega^*_i :=  B_i D^{-1} (D \Omega_i D) D^{-1} B^t_i$ is not changed at all due the projective transform, thus maintaining its rank and PSD nature.  See \cite{Connelly-Whiteley-coning}, Proposition 7, for this same idea applied to a bar framework.  Thus we get the following result.

\begin{theorem} Let $f: \R^d - X \rightarrow \R^d$ be non-singular projective transformation, where $X$ is a $(d-1)$-dimensional affine subspace of $\R^d$, and suppose that for each $i$,  $\p_i \notin X$.  Then for any tensegrity framework, $(G,\p)$ is dimensionally rigid if and only if $(G,f(\p))$ is dimensionally rigid, where the strut/cable designation for $\{i,j\}$ changes only when the line segment $[\p_i, \p_j]$ intersects $X$ and bars go to bars.
\end{theorem}

It is not always true that the universal rigidity of a bar framework is projectively
invariant.  For example, the \emph{orchard ladder}, narrower at the top than at the bottom, as in Figure \ref{fig:orchard-ladder}, is universally rigid, whereas the straight ladder of Figure \ref{fig:ladder} below, a projective image, is flexible in the plane.  
\begin{figure}[here]
    \begin{center}
        \includegraphics[width=0.3\textwidth]{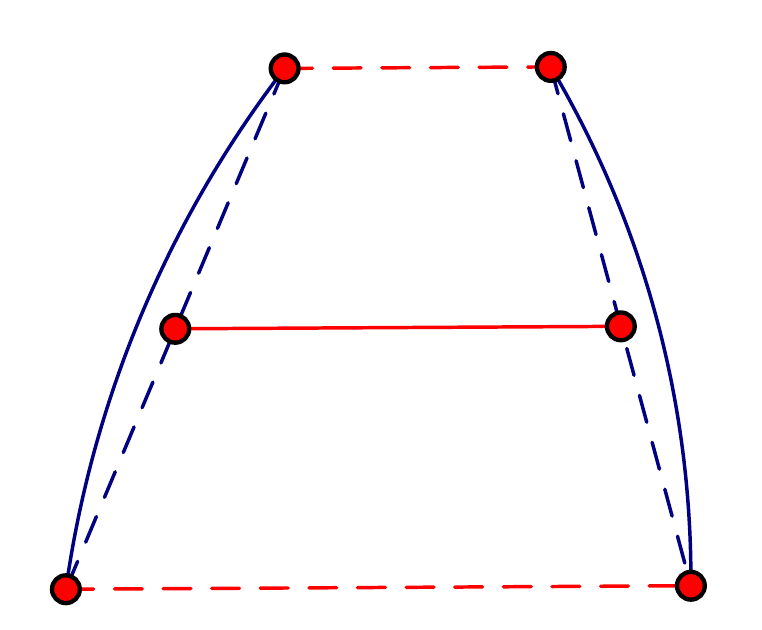}%
        \end{center}
    \caption{This is an example of a universally rigid framework, but the framework of Figure \ref{fig:ladder} below is a projective image that is not universally rigid.  The two poles on the sides are collinear triangles.}
    \label{fig:orchard-ladder}
    \end{figure}


\section{Calculation methods}\label{sect:methods}

We test for dimensional rigidity of $(G,\p)$ by finding the maximal dimension
of any framework $(G,\q)$ that is equivalent to $\p$. This is done by building
up a maximal iterated affine set with an associated iterated PSD stress
as guaranteed by Corollary \ref{cor:mainD}. To do this calculation, we
always maintain a basis matrix $B_i$, where at the start,
$B_0 = I$.

Given $B_{i-1}$ we perform the following steps.

{\bf Find the next stress:}
Look for a matrix $\Omega_i$ such that the restricted stress matrix,
$\Omega^*_i := B_{i-1} \Omega_i B_{i-1}^t$, 
is non zero, PSD and such that
the ``energy" linear constraint
$\p^t (\Omega_i \otimes I^D) \p=0$ holds.
If there is no such solution we are done with the iteration.

\begin{definition}
Given an affine set $\A_{i-1}$ described by a basis matrix $B_{i-1}$.
We say that a restricted stress matrix 
$\Omega^*_i = B_{i-1} \Omega_i B_{i-1}^t$, is
a \emph{restricted equilibrium stress matrix} for $\p$ if 
$P \Omega_i B_{i-1}^t = 0$ holds for $\Omega_i$.
\end{definition}

For any stress matrix $\Omega_i$ that satisfies the energy constraint
and such that the restricted stress matrix, $\Omega^*_i$, is PSD,
we also see that $\Omega^*_i$ must be a restricted equilibrium matrix for
$\p$.
Since we want to get the most milage out
of our linear constraints, we replace the the energy constraint with
this  constraint, which we call a \emph{restricted energy constraint}.

The resulting
problem can be posed as an SDP feasibility problem. If possible
we would like to avoid using an SDP solver, since that is not only expensive,
but, as a numerical algorithm, only approaches, and never exactly hits,
a feasible solution. We discuss this issue more below in Section \ref{sect:computation}.

Sometimes, we can avoid calling an SDP solver
by simply looking at the problem and guessing the correct
$\Omega_i$. For example, if we see, within some two-dimensional
framework,
a degenerate triangle,
(which we will call a pole), it is self evident how to stress that
subgraph.

Another easy case arises when the 
is when the space of solutions for 
$\Omega^*_i$ 
is
only one dimensional. In this case, there is no need to
search for PSD solutions, one only needs to pick one solution
$\Omega^*_i$ and check its eigenvalues.
If it is postive semi-definite, then we have succeeded.
If it is negative semi-definite, then we can negate the matrix, and
we have succeeded. If it is indefinite, then there is no such solution
and we are done with the iteration.

An even easier sub-case of this 
is when the space of 
$\Omega^*_i$ is not only one-dimensional,
but also that the maximal rank of these matrices is $1$. Then we know
immediately that $\Omega^*_i$ is semi-definite.

{\bf Update the basis:}
Given $B_{i-1}$ and a stress $\Omega_i$ we need to update
the basis.
We do this by finding a maximal set of
linearly independent
row vectors of length $n$ that
 is in the row span of $B_{i-1}$, and such that
each of these vectors is in the co-kernel of
$\Omega_i B_{i-1}^t$.

These vectors form the rows of our new basis, $B_i$.
We then
continue the iteration.

{\bf When the iteration is done:}
We simply count the number of rows of the final $B_k$,
which we call $d_k +1$.
If $d_k$ equals $d$, the dimension of the affine span of $\p$,
then we have produced a
certificate  that $\p$ is dimensionally rigid.
Otherwise, we have found a higher dimensional affine set that includes
frameworks equivalent to $\p$ and we have a certificate that
$\p$  is not dimensionally rigid.

\section{Examples}\label{sect:examples}

\subsection{The ladder}\label{subsect:ladder}
\begin{figure}[here]
    \begin{center}
        \includegraphics[width=0.4\textwidth]{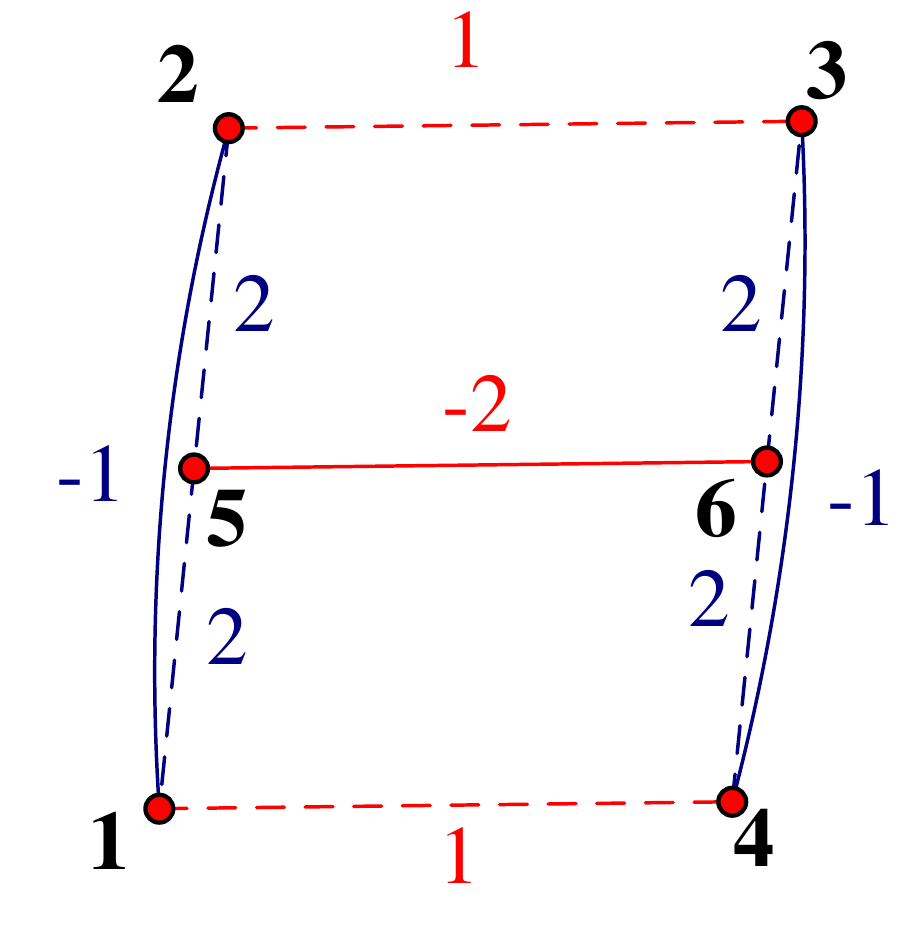}%
        \end{center}
    \caption{This shows a framework with two collinear triangles, each of which provides an affine relation on the space of configurations of the framework $(G,\p)$.  The stresses are indicated and the member connecting the external vertices of poles is indicated by a curved arc. This framework is dimensionally rigid in the plane, but it is not universally rigid, since it has an affine flex in the plane, and since there are only two member directions.  The vertices are labeled in bold.}
    \label{fig:ladder}
    \end{figure}

We first show the process described in Section \ref{subsect:basis-matrix} and Section \ref{sect:methods} applied to the example in Figure \ref{fig:ladder}.  The first level stress matrix, using just the stresses on the vertical members of the ladder, is the following:
\begin{equation*}
\Omega_1 = \begin{pmatrix}
~1 & ~1 	& ~0 & ~0 & -2 & ~0\\
~1  & ~1 	& ~0 & ~0 & -2 & ~0\\
~0  & ~0 & ~1 & ~1 & ~0 & -2\\
~0 & ~0 	& ~1 & ~1 & ~0 & -2\\
-2 & -2 & ~0 & ~0 & ~4 & ~0\\
~0  & ~0 	& -2 & -2 & ~0 & ~4
\end{pmatrix}.
\end{equation*}
This matrix has rank $r_1=2$, a $4$-dimensional kernel, and $d=2$.  The kernel of this matrix defines the affine set $\A_1$.  A basis matrix for $\A_1$ is
\begin{equation*}
B_1 = \begin{pmatrix}
~1 & ~0 	& ~0 & ~0 & 1/2 & ~0\\
~0  & ~1 	& ~0 & ~0 & 1/2 & ~0\\
~0  & ~0 & ~1 & ~0 & ~0 & 1/2\\
~0 & ~0 	& ~0 & ~1 & ~0 & 1/2
\end{pmatrix}.
\end{equation*}

At the second level, we enforce the restricted
equilibrium constraint and find
that the possible candidates for $\Omega^*_2$ must be up to scale, equal to
 
\begin{equation*}
 B_1\Omega_2 B_1^t =  \Omega^*_2=\begin{pmatrix}
~1 & -1 	& ~1 & -1 \\
-1 & ~1 	& -1 & ~1  \\
~1 & -1 	& ~1 & -1 \\
-1 & ~1 	& -1 & ~1 
\end{pmatrix}
\end{equation*}

These are rank 1 with positive trace and thus positive
semi-definite. This $\Omega^*_2$ has an assocated second level
stress $\Omega_2$ where $\omega_{14}=\omega_{23}=1$
and $\omega_{56}=-2$ as in Figure \ref{fig:ladder}. We have $\rank \Omega_1 +\rank \Omega^*_2 = 3 =n-d-1$, making the ladder dimensionally rigid.

\subsection{The $4$-pole Example}\label{subsect:4-pole}

\begin{figure}[here]
    \begin{center}
        \includegraphics[width=0.6\textwidth]{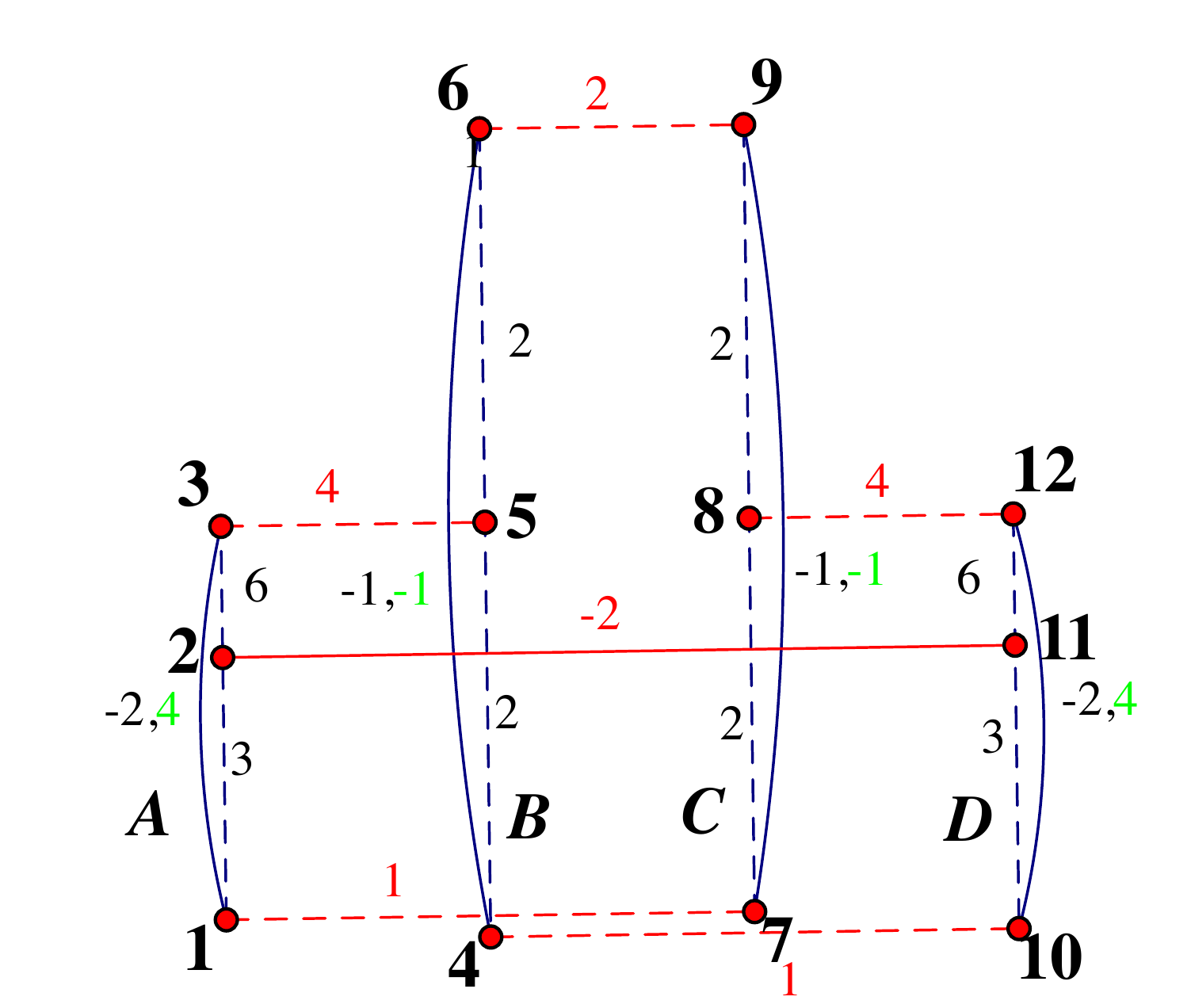}%
        \end{center}
    \caption{This is a framework that is dimensionally rigid in the plane.  Each set of three (nearly) vertical line segments are considered to be a collinear triangle, while the other horizontal members are connected as shown.  This involves at least three levels of iteration as described Section \ref{sect:affine-constraints}, where the levels, in order, are dark blue, red, green.} 
    \label{fig:4-poles}
    \end{figure}
    
 Consider the configuration shown in Figure \ref{fig:4-poles} with four vertical parallel line segments, the \emph{poles}, where each pole is connected to the other three by horizontal members.
 
 The poles are labeled $A, B, C, D$, and the vertices are simply labeled by their number, $1, \dots, 12$. The horizontal spacing between the $AB, BC$, and $CD$ poles is equal.  The vertical spacing of the horizontal members is such that the distance between the $2-11$ line and the $1-7$ line is twice the distance between the $2-11$ line and the $3-5$ line.   The $5$ vertex is the midpoint of the $B$ pole, and the $8$ vertex is the midpoint of the $C$ pole.
 The stresses on this framework are as indicated.  These are simply arranged so that the lever arm moments are all $0$.  The question is whether the appropriate stress matrices are PSD of the right rank.  
 
 The stress for each pole is rank one and they can all be combined to one rank $4$ stress, which can be considered as a stress at the first level.  It is simply the certificate, in any equivalent framework, that each pole remains straight maintaining the ratio of each of the lengths.  The stress for each of those members is proportional to the reciprocal of its length in absolute value.  The stress for the longest member of each collinear triangle is negative, while the other two are positive.  
 

 One can then choose a basis for $\A_1$ and search for the 
restricted equilibrium matrices as in Section \ref{sect:methods}. It turns out that, as in the ladder example, the space of a possible
equilibrium $\Omega^*_2$ is only one-dimensional, and these
have rank $4$. We check and find that these matrices are semi-definite.
An associated $\Omega_2$ is shown in red in Figure \ref{fig:4-poles}. 
We then choose a basis for $\A_2$ and use the methods of Section \ref{sect:methods} one final time. Again, we find a one-dimensional space
of equilibrium matrices $\Omega^*_3$, and these have rank $1$.
An associated $\Omega_3$ can be constructed with
$\omega_{1,3} = \omega_{10,12} = 4$
and 
$\omega_{4,6} = \omega_{7,9} = -1$.

The sum of the ranks  is $4 + 4 +1 = 9 = 12 - (2 + 1)  = n - ( d+1)$, so this framework is dimensionally rigid in the plane.
It is not universally rigid since the original framework has only
two member directions.

One interesting feature of this example is that
the stress $\Omega_2$ involves all of the vertices of the graph $G$ from the second level, and yet it still needs another level for the complete analysis of its dimensional rigidity.  

The first stage in this example involves only the four collinear triangles, which imply the corresponding affine constraints on the the configuration.  Suppose one initially starts with those four affine constraints and then proceeds with the analysis, where the distance constraints on the poles is dropped?  It turns out that the configuration is not dimensionally rigid in the plane, since at the third level the member constraints in the poles are needed again.  The maximal dimensional realization, in that case, is $\R^3$.    

\subsection{The $4$-pole Extended Example}\label{subsect:4-pole-extended}

\begin{definition}A \emph{spider web} is a tensegrity, where some subset of the vertices are fixed, and all the members are cables.
\end{definition}

  For a spider web, it was shown in \cite{Connelly-energy} that it is locally rigid if and only if it is universally rigid, and that when it is universally rigid the iterated construction simplifies to a sequence of proper subgraphs, where the number of vertices decreases at each stage as in Figure \ref{fig:iterated-rigid}.  Another  example of the iteration process is shown in Figure \ref{fig:general-position}, where the vertex $A$ is added at the second stage.  In each of those examples, there is a proper subgraph that is universally rigid on its own without using the presence of the other vertices.   
  
  Figure \ref{fig:4-pole-extended} shows that, in general, when the framework is universally rigid in more than one step of the iteration, there may be no proper sub-framework that is universally rigid on its own.  The stresses at each level are shown.
   \begin{figure}[here]
    \begin{center}
        \includegraphics[width=0.9\textwidth]{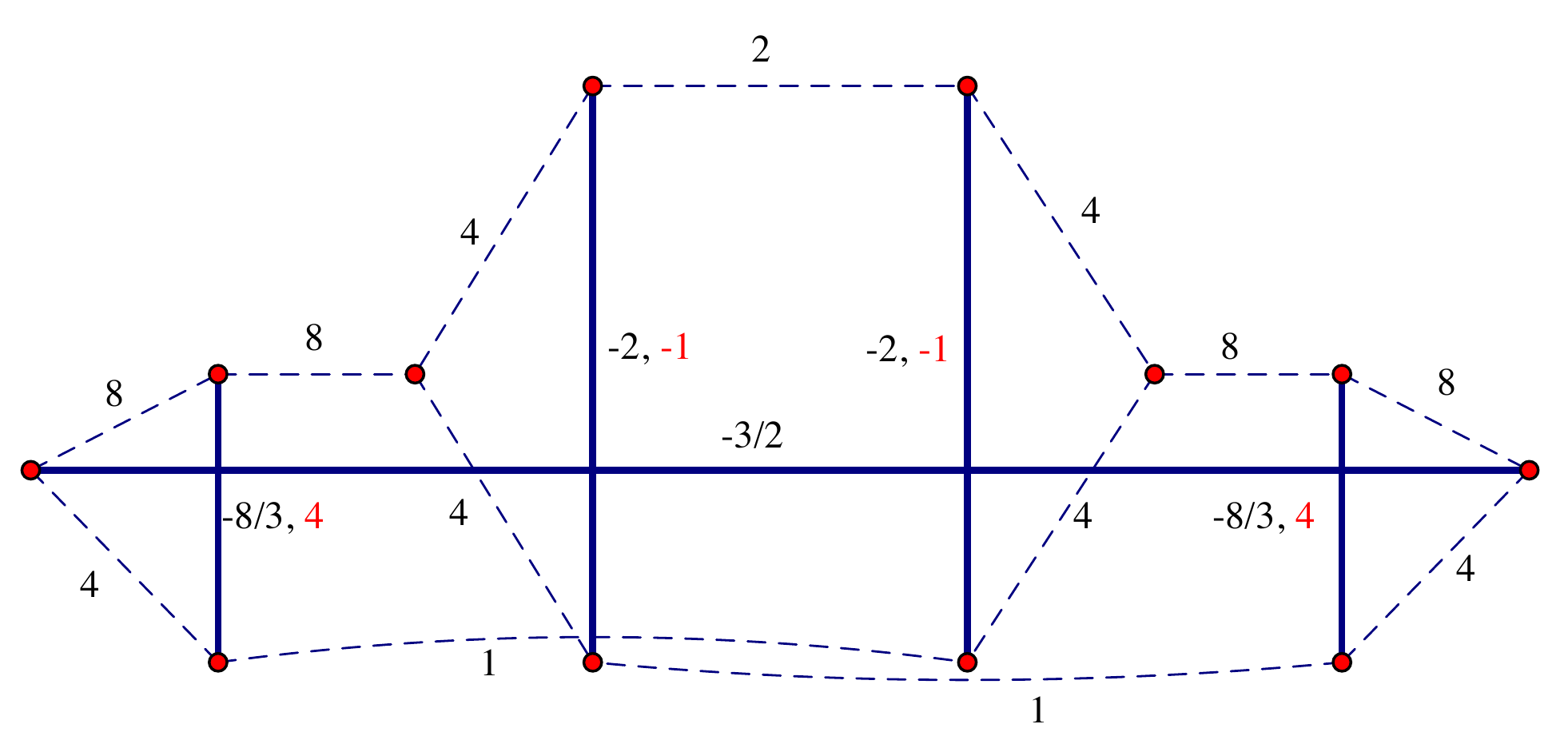}%
        \end{center}
    \caption{This is an example of a universally rigid tensegrity framework in the plane that has only one stress that is PSD of rank $8$, one less than the maximal possible $n -d-1=12 - 2 -1 =9$.  There is a stress at the second stage which is PSD of rank one in the affine set defined by the stress at the first stage.  The vertices of this configuration are the same as those in Figure \ref{fig:4-poles}, except the interior point of each pole has been moved half the distance (left or right as indicated) between adjacent poles.}
    \label{fig:4-pole-extended}
\end{figure}

This is a perturbed version of Figure  \ref{fig:4-poles}, and it turns out to be universally rigid by the process described here, but using only two stages instead of three as in Subsection \ref{subsect:4-pole}.  Since the stressed members have more than two directions in the plane, and since it is dimensionally rigid in the plane as with Figure  \ref{fig:4-poles}, it is universally rigid.

In both of these cases, we were able to find the certifying sequence of
stresses without calling a PSD solver. This was because, at each 
step, there was only a one-dimensional space of 
restricted equilibrium matrices
$\Omega^*$ as candidates. Since they were rank $1$, we automatically
knew that they were semi-definite, and for the second step, for the $4$ poles,
we just checked that it was of rank $4$.

More generally, if we end up with a higher dimensional space of 
equilibrium $\Omega^*$ as candidates, we might have a harder time
determining if that space includes a positive semi-definite one.
We discuss this more below in Section \ref{sect:computation}.

\subsection{A hidden stress}\label{subsect:hidden}

One of the problems with SDP is finding even one PSD equilibrium
stress (or more generally restricted  equilibrium stresses at later
stages).  The following example is a framework, where the dimension of
PSD equilibrium stresses is a low dimensional subcone of the space of
all equilibrium stresses.

\begin{figure}[here]
    \begin{center}
        \includegraphics[width=0.4\textwidth]{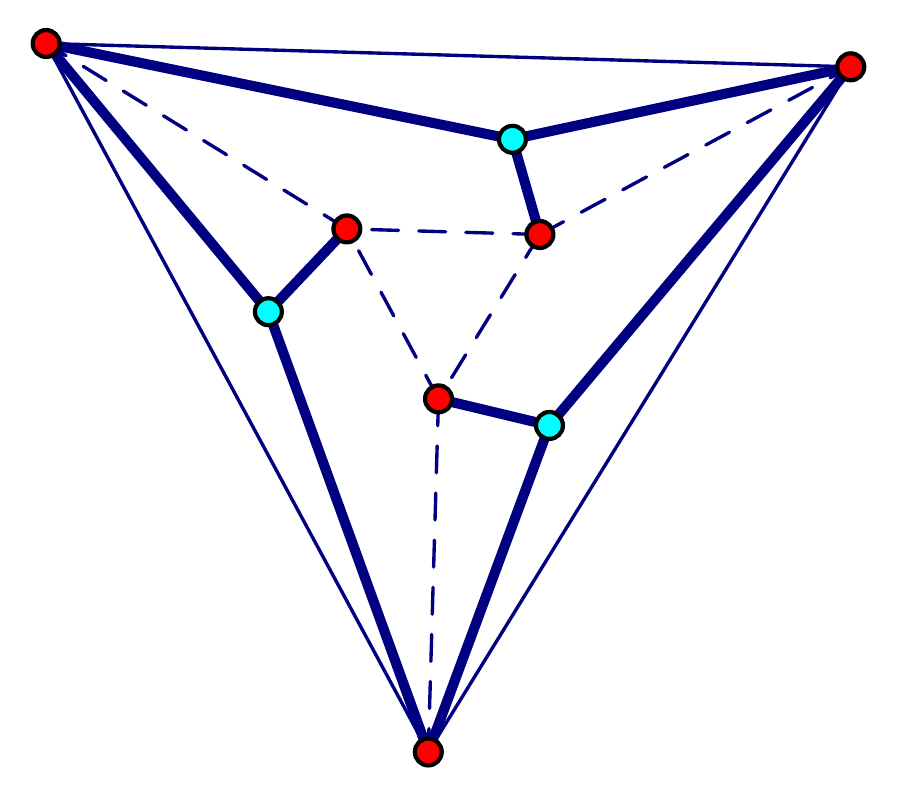}%
        \end{center}
    \caption{This is an example of a universally rigid bar framework in the plane that has a three-dimensional space of equilibrium stresses but only a one-dimensional space that is PSD.}
    \label{fig:non-open}
\end{figure}

 The two triangles and the members joining corresponding vertices constitute a super stable PSD subframework as in Figure \ref{fig:general-position}.  Since the whole (bar) framework is infinitesimally rigid in the plane, and that there are $18$ members and $9$ vertices, the dimension of the stress space is $18-2\cdot 9 + 3=3$.  Equilibrium at each blue vertex implies that the three stresses at a blue vertex must all have the the same sign.  But any equilibrium stress, non-zero on any of the members adjacent to the blue vertices, cannot be all have the same sign for all the members adjacent to all the blue vertices.  This is because the twisting infinitesimal motion of the inner triangle relative to the outer triangle either decreases all the members adjacent to the blue vertices or increases them all.  So one of the set of three members adjacent to a blue vertex has to have all negative stresses.  This stress cannot be PSD since by moving that single blue vertex the stress energy must decrease.
 
\section{Computational  matters}\label{sect:computation}

An important property of universal rigidity is that often it can be
calculated efficiently using various SDP
algorithms.  For example, see \cite{Alfakih-Wolkowicz, Ye-talk,
Pataki, Ye-So-tensegrity, Wolkowicz-facial, Borwein-Wolkowicz, Ramana}
for information on this vast subject including facial reduction.  In
particular, if one is given the edge lengths $\e$ for a graph $G$, one
can use SDP to find a configuration $\p$ whose edge lengths
approximate $\e$.  More precisely, an $\epsilon$-approximate
configuration $\p$ can be found,  in some unconstrained dimension $D$ 
if it exists, in 
time  polynomial in $\log(1/\epsilon)$, where
$n$ is the number of vertices of $G$,
and $m$ is the number of members of $G$, 
as
described in \cite{Ye-talk}.
So this can be used to attempt to see if
the existence problem is feasible and to
attempt to find a
satisfying configuration when it is feasible.

But, as mentioned in Section \ref{sect:introduction}, one problem is
that even though the member lengths of the approximation are close to
the given lengths, the configuration may be quite a distance from one
implied by the actual constraints.  Small errors in the edge lengths
can imply large errors in the proposed configuration as in the
framework in Figure \ref{fig:iterated-rigid}, but see \cite{Javanmard}.  In principle, one could
use the calculation as evidence that a given configuration is
universally rigid in $\R^2$, but Figure \ref{fig:iterated-rigid-2}
shows that it may appear that $(G,\p)$ has equivalent configurations in
$\R^3$ or higher, even with $\epsilon > 0$ is very small.

In contrast to this ``primal appoach", we have shown in this paper that
when a framework is dimensionally or universally rigid, there must exist
a certificate, in the form of an iterated PSD stress, that
conclusively proves the dimensional or universal rigidity of the
framework.  

Although finding these stresses also involves solving an
SDP problem,
in many cases, though we can hope to \emph{exactly} solve this ``dual"
SDP.  At any level of the analysis here, there is a linear space of
restricted equilibrium stress matrices $\Omega^*_i$ as described in
Section \ref{sect:methods}.  If there is such a PSD matrix of
maximal rank among all such $\Omega^*_i$, then the PSD 
restricted equilibrium
stresses includes an open subset of the space of all restricted equilibrium
stresses.  In this case, it reasonable to expect that we can exactly
find such a solution.  Thus, even if the numerical solution from an SDP
solver is, say, PSD but not quite in 
restricted equilibrium, a sufficiently close
restricted equilibrium stress will still be PSD and of maximum rank.

In fact this ``maximal rank case" must always occur in the last step of our
iterated process 
so, for example, if the framework $(G,\p)$ is super
stable (in other words, there is only one step in the iterated process
described here), then the PSD solutions are
full dimensional within the linear space of equilibrium stress
matrices.  This is the situation if $\p$ is generic in $\R^d$, and the
framework $(G,\p)$ is universally rigid, since this must be super
stable by Theorem \ref{thm:Gortler-Thurston2}.  The two
examples on the left in Figure \ref{fig:tensegrity-examples} have that
property.

In other cases, though we may not be able to \emph{exactly} solve this
``dual" SDP, the example of Figure \ref{fig:non-open} shows a case where
the PSD  equilibrium stresses are all of lower rank than
the indefinite equilibrium stresses, and thus do
NOT form an open subset of
the space of equilibrium stresses.
If the dimension of PSD matrices is lower than
the dimension of all the equilibrium matrices, then
we may have to
resort to using the SDP to ``suggest" what an actual PSD matrix is
(since it will only converge to a
PSD matrix in the limit).

More generally,
when the configuration is not generic, you
have to ask: how is the configuration even defined?  It is possible to
create configurations precisely so that they become universally rigid.
For example, the symmetric tensegrities of many artists are created in
such a way that they become super stable, but not at all generic, not
even infinitesimally rigid, even though they are super stable.
Indeed, they often have certain symmetries that can be used to
simplify the calculations and create tensegrities that are super
stable.  The representation theory of some small finite groups can be
exploited to create these configurations.  A brief explanation is in
\cite{Connelly-Back}.  This is called \emph{form finding} in the
Engineering literature, as in \cite{Skelton-form, Schek}.

Stresses and iterated stresses might also be useful during the process
of calculating a realization $\p$ from an input graph $G$ and input
set of edge lengths $\e$. 
Note though, when we are just given
input lengths and are searching for an appropriate $\Omega$, we do not
have enough information to express the (restricted) equilibrium linear
constraint and can only use the ``energy linear constraint": $0=
\sum_{i<j} e_{ij}^2 \omega_{ij}$. Therefore, we do not expect to be in
a ``maximal rank" setting.  
Once we have computed the iterated stresses,
then we just need to look for $\p$ within the final affine set.  
As described in the appendix in
\cite{Gortler-affine-rigidity}, when $\p$ is universally rigid, this
calculation of $\p$ within its affine set can be done easily
by solving a 
certain small linear system.
(In the case
that $\p$ is not universally rigid but is only dimensionally rigid,
then that linear system will be singular. Still, since we 
have restricted ourselves to the correct affine set, 
we only need to solve  small SDP problem, which must be applied over the space of
$(d+1)$-by-$(d+1)$ matrices). 

In addition to the example in
\cite{Connelly-Back}, a graph coloring problem can be solved using
this idea as in \cite{Pak-color}.


\section{Extensions}\label{sect:protocol}

In general, we propose the following procedure for determining/creating universally rigid frameworks and tensegrities. 
First a (tensegrity) graph $G$, and a corresponding configuration $\p$, is defined.  A priori, a sequence of affine sets in configuration space can be given as well, as in Section \ref{sect:affine-constraints}.   These sets may or may not be a consequence of the geometry of the configuration $\p$.  Then at each stage, one either calculates a PSD stress for the given configuration or one assumes that there is a corresponding affine constraint.  If the constraints are consistent, then one has a proof that the configuration is dimensionally rigid or universally rigid, depending on the stressed member directions.  For example, if there appears to be a (proper) PSD stress for a given affine set, one can assume that it exists and proceed, getting further affine sets.  It would depend on the circumstance as to whether the particular affine constraint is reasonable or not.  For example, in Figure \ref{fig:iterated-rigid}, one might suspect that the eight subdivided vertical members are straight, but initially not the others.  Only then might one suspect that the four smaller horizontal members are straight, etc.  After this is finished one can conclude that the whole framework is universally rigid.

The idea of assigning nested affine constraints is a generalization of the idea of a body-and-bar framework as defined by Tay and Whiteley in \cite{Tay-Whiteley, Tay-bodies}.   The concept of nested affine sets, introduced here, is closely related to the concepts of hypergraphs of points and affine rigidity introduced in \cite{Gortler-affine-rigidity}.  Also, a recent result in \cite{Connelly-Jordan-Whiteley-2013} shows that body-and-bar frameworks are generically globally rigid in $\R^d$ if they are generically redundantly rigid in $\R^d$.  

\begin{definition} \emph{Redundant rigidity} means that the framework is locally rigid, and remains so after the removal of any member.  
\end{definition}

It is also true \cite{Connelly-Jordan-Whiteley-2013} that such body-and-bar graphs always have a generic configuration that is universally rigid in $\R^d$ as well. 
 
\section{Possible future directions and questions}\label{sect:future}

It is also possible to use stresses to estimate the possible perturbations of a given tensegrity or framework. The sign of a PSD stress associated to each member corresponds to an inequality constraint.  If all of those constraints are such that at least one of the constraints is violated, we know that the edge length perturbed configuration cannot be achieved.  This imposes somewhat weak, but useful, conditions on which sets of members can increase or decrease in length.  If there are more PSD stresses on the members, there will be more of these sign constraints that can be calculated even if the tensegrity framework is not rigid.

One could use universal rigidity properties to understand flexible structures by adding members providing parameters for controlling the motion of a flexible framework.  For a fixed length of such additional members, the configuration could be determined.  As that length varies the whole configuration could flex in a controlled way.  

For the case of generic global rigidity, the notion of \emph{globally linked} pairs of vertices is discussed in \cite{Jordan-Jackson-linked, Jordan-Jackson-unique}. This means that although the whole framework may not be globally rigid, some pairs of vertices would be forced to have a fixed length for all equivalent configurations in the same dimension.   A similar question in the universally rigid category involving configurations in higher dimensions that satisfy the tensegrity inequality constraints would be interesting to explore. 

Even to determine whether a framework is universally rigid on the line is interesting.  In \cite{Jordan-universal-line} it is determined when a rigid one-dimensional complete bipartite bar-and-joint framework in the line is universally rigid, as well as several open questions in this direction.  We have a forthcoming paper that extends this result, and determines when any complete bipartite framework in any dimension is universally rigid.


The weavings of \cite{Whiteley-Polarity-I, Whiteley-Polarity-II, Pollack-weaving, Pollack-weaving-II} concern lines in the plane that may or may not arise from projections of configurations of lines in a higher dimension.  Particularly, there is a relation to stresses of dual configurations in \cite{Whiteley-Polarity-I, Whiteley-Polarity-II}.  Can there be a connection to the poles in universal rigidity?

\section{Acknowledgement}\label{sect:acknowledgement}

We would like to thank Dylan Thurston for countless helpful
conversations on convexity.

\bibliographystyle{plain}
\bibliography{NSF-10,framework}

\end{document}